\documentclass[
a4paper,				
english
]{extarticle}

\newif\ifshow
\showtrue   

\makeatletter
\title{A Harmonic Framework for Vector Fields and Differential Operators on SO(3)}\let\Title\@title
\date{\today}\let\Date\@date
\makeatother
\newcommand{\keywords}{rotation group SO(3),harmonic analyis,Wigner-D functions,generalized spherical harmonics,tangent space basis,tangent bundle frame,harmonic vector fields on SO(3),differential operators,optimal sampling,rotational diffusion}

\usepackage{header}
\usepackage{headerFormulas}

\usepackage[final]{microtype}	

\begin{document}

\setcounter{section}{0}
\setcounter{subsection}{1}

\author{Ralf
  Hielscher\orcidlink{0000-0002-6342-1799}%
  \thanks{ralf.hielscher@math.tu-freiberg.de}~,~%
  Erik Wünsche \orcidlink{0000-0001-5178-6654}%
  \thanks{erik.wuensche@math.tu-freiberg.de}}

\maketitle

\begin{abstract}
  We present a comprehensive framework for tangent vector fields and differential operators on the rotation group $\SO3$ using harmonic series expansions, which provides a mathematical foundation for their implementation in the crystallographic texture analysis software \texttt{MTEX}~\cite{MTEX}.
  The central idea is to employ the standard left- and right-invariant frames as global orthonormal frames of the tangent bundle, thereby avoiding the numerical instabilities of the classical tangent space basis derived from the Jacobian of the Euler angle parametrization.

  While the choice of frames is primarily motivated by their geometric and numerical properties, their full potential emerges in the harmonic setting.
  Representing tangent vector fields through harmonic expansions of their frame components, we derive explicit frequency domain formulas for the gradient, divergence, and curl, allowing these operators to be applied directly to the harmonic coefficients.
  Moreover, we show that these differential operators preserve harmonic band-limitedness.

  In addition, the left- and right-invariant representations of tangent vector fields can be transformed into one another directly in the frequency domain, with an increase in harmonic bandwidth of at most one degree.
\end{abstract}



\section{Introduction}

Many problems involving rotations are naturally formulated in terms of scalar functions on the rotation group $\SO3$.
Such functions may represent rotational probability densities~\cite{Morawiec2004,Mason2009,Eghtesad2018}, interaction or potential energies~\cite{Tarroni1991}, orientation-dependent material properties~\cite{Mainprice2011,Bunge1989,Kilian2017}, correlation functions~\cite{Hielscher2019,Lenthe2020}, or objective functions in optimization problems on $\SO3$~\cite{Graef2013,Dick2023}.

For scalar functions, harmonic analysis on $\SO3$ provides a well-established analogue of Fourier analysis, cf.~\cite{Vilenkin2012,Bunge1982}:
The Wigner-D functions form an orthonormal basis of $\L2SO3$, and band-limited functions admit finite harmonic expansions.
This allows fast algorithms~\cite{Potts2009,Kostelec2008,Risbo1996,Hielscher2026} to be used for the efficient approximation, evaluation, and manipulation of scalar functions on $\SO3$.
The MATLAB toolbox \texttt{MTEX}~\cite{MTEX} is a fundamental tool in crystallographic texture analysis and provides a comprehensive computational framework for working with such harmonic representations, including implementations of the corresponding fast algorithms.

However, many operations of interest naturally lead beyond scalar functions to tangent vector fields on $\SO3$.
The Riemannian gradient of a scalar function on $\SO3$ is a tangent vector field~\cite{Varshalovich1988}, optimization methods evolve rotations along gradient directions~\cite{Graef2013,Dick2023}, and evolution equations on $\SO3$ are often formulated in terms of tangent fluxes and their divergence, cf.~\cite{Tarroni1991,Morawiec2004,Bunge1984}.
Thus, a harmonic framework for scalar functions on $\SO3$ naturally leads to the question of how tangent vector fields and differential operators should be represented in this harmonic setting.

A tangent vector field $\vec F$ on $\SO3$ assigns to each rotation $\mat R\in\SO3$ a tangent vector $\vec F(\mat R)\in T_{\mat R}\SO3$, where the tangent space itself varies with $\mat R$.
To describe tangent vectors at different rotations in a consistent manner, we therefore need to choose a basis in each tangent space $T_{\mat R}\SO3$ such that these bases vary smoothly with $\mat R$.
Such a smoothly varying family of tangent space bases is called a frame of the tangent bundle, cf. \Cref{sec:TangentBundleFrames}.
More precisely, with respect to a frame $\set{\vec X_1,\vec X_2,\vec X_3}{}$ of $\TSO3$, the vector field admits the unique representation
\[ \vec F (\mat R) = F_{1}(\mat R) \vec X_{1}(\mat R)  + F_{2}(\mat R)\vec X_{2}(\mat R) + F_{3}(\mat R) \vec X_{3}(\mat R), \]
where $F_1,F_2,F_3\colon\SO3\to\IR$ are scalar component functions.
Once a frame has been fixed, any smooth tangent vector field can be represented harmonically by expanding its three scalar component functions in the Wigner-D basis.
Different frames therefore lead to different harmonic representations of the same geometric vector field.
The central question is therefore to identify a geometrically natural frame with respect to which the relevant differential operators admit simple and explicit frequency domain representations.

Rotations are commonly represented by Euler angles, which are also used to express the Wigner-D functions and the Haar measure on $\SO3$.
Thus, a natural first choice is the local frame induced by the Euler angle parametrization, see \Cref{sec:TangentSpaceBasisEulerAngle}.
This frame is widely used in the literature, cf.~\cite{Bunge1986,Kalidindi2005,Graef2013}, and is obtained by differentiating the parametrization with respect to its three Euler angles.
However, it is not globally valid, since it degenerates at the singular sets of the Euler angle parametrization, which makes it numerically ill-conditioned near them.

As an alternative, in~\Cref{sec:TangentSpaceBasisStandard}, we consider two standard frames, previously used in~\cite{Tarroni1991,Boscain2005,Breev2023}, which are globally well-defined on $\SO3$ and yield an orthonormal basis of each tangent space $\TRSO3$.
They are obtained by translating a fixed orthonormal basis of the Lie algebra $\mathfrak{so}(3)$ to each tangent space using the two equivalent representations
\[ \TRSO3 = \mat{R}\cdot\mathfrak{so}(3) = \mathfrak{so}(3)\cdot\mat{R}, \]
corresponding to left- and right-translation, respectively.

The main contribution of this paper is a harmonic framework for tangent vector fields on $\SO3$ with respect to these two global frames, which provides the mathematical foundation for the corresponding functionality in \texttt{MTEX}.
In \Cref{lem:FrameChangeBandWidth}, we analyze the transformation between the left- and right-invariant representations of tangent vector fields as a linear operator in the frequency domain and show that, under this transformation, the harmonic bandwidth increases by at most one degree.
Both representations therefore remain compatible with finite harmonic approximations and can be used interchangeably depending on the structure of the underlying problem.
Furthermore, in \Cref{cor:GradientHarmonicSeries} and \Cref{lem:DivergenceCurl}, we derive explicit frequency domain formulas for the gradient, divergence, and curl.
These formulas allow the harmonic coefficients of the gradient component functions to be computed directly from those of the underlying scalar function, and the coefficients of the divergence and curl directly from those of the component functions of the associated vector field.
In contrast to the Euler angle induced frame, the global orthonormal left- and right-invariant frames preserve band-limitedness under these differential operators, as shown in \Cref{cor:GradientBandLimited}.

At the end of the paper, we apply the developed theoretical framework to two numerical problems.
More precisely, in~\Cref{sec:Compactification}, we address the approximation of density functions on $\SO3$ by small sets of representative rotations, formulated as an optimization problem and solved using gradient descent. This requires the efficient computation of gradients on $\SO3$.
Finally, in~\Cref{sec:BrownianMotion}, we simulate anisotropic rotational diffusion arising from Brownian motion, thereby demonstrating the modular use of the framework for the numerical simulation of evolution equations on $\SO3$ through the frequency domain evaluation of the differential operators.
Both optimal sampling and the simulation of rotational density evolution are important applications in crystallography, cf.~\cite{Eghtesad2018,Knezevic2015} and~\cite{Bunge1984,Bunge1986}, respectively, with rotational diffusion providing a drift-diffusion extension of the transport equations arising in crystal-plasticity simulations.


\section{Preliminaries}

The special orthogonal group in $\IR^3$ is defined as
\begin{equation*}
  \SO3 = \set{\mat R \in \IR^{3\times3} }{\mat R^T\mat R=\id \text{ and } \det(\mat R)=1}.
\end{equation*}
It forms a compact Lie group and can be regarded as a three-dimensional Riemannian manifold.
Every rotation $\mat R_{\vec n}(\omega) \in \SO3$ can be represented, not necessarily uniquely, by a rotation axis $\vec n\in\S{2}$ and a rotation angle $\omega \in \IT=\IR/(2\pi\IZ)$.
Furthermore, rotations can be represented using the Euler angle parametrization
\begin{equation*}
  \epsilon\colon \IT\times[0,\pi]\times\IT\to\SO3, \quad (\alpha,\beta,\gamma)\mapsto \mat R_{\vec z}(\alpha) \, \mat R_{\vec y}(\beta) \, \mat R_{\vec z}(\gamma).
\end{equation*}
This parametrization is not one-to-one and is singular at \(\beta=0\) and \(\beta=\pi\).

As a compact Lie group, $\SO3$ carries the bi-invariant Haar measure $\mu$, normalized such that $\mu(\SO3) = 8\pi^2$.
This measure naturally defines the Hilbert space $\L2SO3$, equipped with the inner product
\[ \scp{f_{1}}{f_{2}}_{\L2SO3} = \frac1{8\pi^{2}}\int_{\SO3}f_{1}(\mat q) \conj{f_{2}(\mat q)}\d{\mu(\mat q)}, \qquad \text{for all } f_{1},f_{2}\in\L2SO3 \]
and the corresponding norm $\norm{\cdot}_{\L2SO3}$.

\begin{definition}\label{def:WignerdD}
  Let $l \in \IN_0$ and $k,k' \in \IZ$ with $|k|,|k'| \leq l$.

  The $\Lp{2}$-normalized Wigner-D function of degree $l$ and orders $k,k'$ is the complex-valued function~$\WignerD{l}{k}{k'} \colon \SO3 \to \IC$ defined on a rotation $\mat R(\alpha,\beta,\gamma)$ via
  \begin{equation*}
    \WignerD{l}{k}{k'}(\mathbf{R}(\alpha,\beta,\gamma)) \coloneqq \sqrt{2l+1} \, \e^{-\i \, k \, \alpha} \, \Wignerd{l}{k}{k'}(\cos\beta) \, \e^{-\i \, k' \, \gamma}.
  \end{equation*}
  Here, $\Wignerd{l}{k}{k'} \colon [-1,1] \to \IR$ is the Wigner-d function, which depends only on the Euler angle $\beta$ and is explicitly given by
  \begin{equation*}
    \Wignerd{l}{k}{k'}(x) \coloneqq (-1)^{\nu} \binom{2n-s}{s+a}^{\frac12} \binom{s+b}{b}^{-\frac12} \left(\frac{1-x}{2}\right)^{\frac{a}2} \left(\frac{1+x}{2}\right)^{\frac{b}2} P_s^{a,b}(x),
  \end{equation*}
  with
  \[ a=|k-k'|, \quad b=|k+k'|, \quad s=l-\max\{|k|,|k'|\}, \quad \nu = \charFunc{k>k'} \cdot(k+k'), \]
  and where $P_s^{a,b}$ denotes the Jacobi polynomial of degree $s$ with parameters $a$ and $b$, see~\cite{Szegoe1975}.
\end{definition}

The following Lemma, which concerns the derivative of the Wigner-d functions, will be required later in the proof of \Cref{thm:Gradient}.
\begin{lemma}\label{lem:DerivativeWignerd}
  It holds that
  \begin{equation*}
    \frac{\d}{\d{\beta}} \Wignerd{l}{k}{k'}(\cos\beta)\Big|_{\beta=0} =
    \begin{cases}
      \frac12 \sqrt{(l+k+1)(l-k)}, 	& \text{if } k+1=k', \\
      -\frac12 \sqrt{(l-k+1)(l+k)}, 	& \text{if } k-1=k', \\
      0,															& \text{otherwise}.
    \end{cases}
  \end{equation*}
\end{lemma}
\begin{proof}
  Assume first that $k+k'$ is even.
  In this case, $\Wignerd{l}{k}{k'}(\cdot)$ is a polynomial, and thus the composition $\Wignerd{l}{k}{k'}\circ\cos$ is differentiable.
  Moreover, using the symmetry property of the Wigner-d functions,
  \[ \Wignerd{l}{k}{k'}(\cos(-\beta)) = (-1)^{k+k'}\,\Wignerd{l}{k}{k'}(\cos\beta), \qquad \text{see~\cite{Varshalovich1988}}, \]
  it follows that $\Wignerd{l}{k}{k'}\circ\cos$ is an even function, and therefore its derivative at $\beta=0$ vanishes.

  Now assume that $k+k'$ is odd.
  By the same symmetry property, $\Wignerd{l}{k}{k'}\circ\cos$ is then an odd function.
  Differentiating via the central difference quotient yields
  \[ \frac{\d}{\d{\beta}} \Wignerd{l}{k}{k'}(\cos\beta)\Big|_{\beta=0} = \lim_{\beta\to0} \frac{\Wignerd{l}{k}{k'}(\cos\beta) - \Wignerd{l}{k}{k'}(\cos(-\beta))}{2\beta} = \lim_{\beta\to0} \frac{\Wignerd{l}{k}{k'}(\cos\beta)}{\beta}. \]

  Let $a,b,s,\nu$ be as in \Cref{def:WignerdD}, and define
  \[ C \coloneqq (-1)^{\nu} \binom{2n-s}{s+a}^{\frac12} \binom{s+b}{b}^{-\frac12}. \]
  By assumption, both $a$ and $b$ are odd, so that $\frac{a-1}2$ and $\frac{b-1}2$ are non-negative integers.
  Then
  \begin{equation*}
    \lim_{\beta\to0}\frac{\Wignerd{l}{k}{k'}(\cos\beta)}{\beta} = C \cdot \lim_{\beta\to0} \frac{\sin\beta}{2\beta} \left(\frac{1-\cos\beta}{2}\right)^{\frac{a-1}2} \left(\frac{1+\cos\beta}{2}\right)^{\frac{b-1}2} P_s^{a,b}(\cos\beta)
  \end{equation*}
  vanishes if $a=\abs{k-k'} \neq 1$.

  For the case $a=\abs{k-k'}=1$, we obtain
  \begin{equation*}
    \lim_{\beta\to0}\frac{\Wignerd{l}{k}{k'}(\cos\beta)}{\beta} = \frac{C}2 \cdot \binom{s+1}{s} = (-1)^{\nu} \frac12 \sqrt{\braces*{n+\frac{b+1}2}\braces*{n-\frac{b+1}2+1}}.
  \end{equation*}
  Finally, distinguishing the two cases $k'=k+1$ and $k'=k-1$ completes the proof.
\end{proof}

Omitting the $\sqrt{2l+1}$ normalization, the Wigner-D functions coincide with the matrix elements of the irreducible unitary representations of $\SO3$, see~\cite{Vilenkin2012}. Hence they satisfy
\begin{equation}\label{eq:RepresentationProperty}
  \WignerD{l}{k}{k'}(\mat R \mat Q) = \frac{1}{\sqrt{2l+1}} \sum_{j=-l}^l \WignerD{l}{k}{j}(\mat R) \, \WignerD{l}{j}{k'}(\mat Q).
\end{equation}

According to the Peter-Weyl theorem, the Wigner-D functions constitute a complete orthonormal basis of $\L2SO3$.
Consequently, every function $f\in\L2SO3$ admits a unique harmonic expansion of the form
\begin{equation*}
  f(\mat R) = \sum_{l=0}^{\infty}\sum_{k,k'=-l}^l \fhat{l}{k}{k'} \, \WignerD{l}{k}{k'}(\mat R),
\end{equation*}
where $\fhat{l}{k}{k'}  = \scp[\big]{f}{\WignerD{l}{k}{k'}}_{\L2SO3}$ are the harmonic coefficients of $f$.
Furthermore, the subspace of $L$-band-limited functions is defined as
\[ \Pi_{L}(\SO3) \coloneqq \mathrm{span}\set*{\WignerD{l}{k}{k'}}{(l,k,k')\in\J{L}}, \]
where $\J{L}\coloneqq\set{(l,k,k')}{l=0,\dots,L \text{ and } k,k'=-l,\dots,l}$.


\section{Tangent Bundle Frames on $\SO3$}\label{sec:TangentBundleFrames}

Since $\SO3$ is a Lie group, its associated Lie algebra
\[ \mathfrak{so}(3) = \set{\mat S \in\IR^{3\times3}}{\mat S = -\mat S^{\top}} \]
consisting of skew symmetric $3\times3$ matrices represents the geometric tangent space at the identity, i.e. $\mathcal{T}_{\id}\SO3=\mathfrak{so}(3)$.

Moreover, the Lie group structure provides two equivalent representations of the tangent space at an arbitrary rotation $\mat R\in\SO3$, obtained via the left or right translation of the Lie algebra.
Using left translation, the tangent space can be written as
\begin{equation}\label{eq:TangentSpace}
  \TRSO3 = \mat R \cdot \mathfrak{so}(3) = \set{\mat R \cdot \mat S}{\mat S \in \mathfrak{so}(3)}.
\end{equation}
Using right translation, one obtains the equivalent representation
\[ \TRSO3=\mathfrak{so}(3) \cdot \mat R.\]

Since $\SO3$ is a Riemannian manifold, it is equipped with a Riemannian metric $g$ that assigns to each $\mat R\in\SO3$ an inner product $g_{\mat R}\colon \TRSO3 \times \TRSO3 \to \IR$ defined by
\[ \RieMetric{\mat T_{1}}{\mat T_{2}}  = \frac12 \scp{\mat T_{1}}{\mat T_{2}}_{\mathrm F},\]
where $\scp{\cdot}{\cdot}_{\mathrm{F}}$ denotes the Frobenius inner product.

The tangent bundle of $\SO3$ is defined by
\[ \TSO3 \coloneqq \bigcup_{\mat R\in\SO3} \{\mat R\}\times\TRSO3. \]

\begin{definition}
  A smooth vector field on $\SO3$ is a smooth section $\vec F\colon\SO3\to\TSO3$ of the tangent bundle, i.e.,
  \[ \vec F(\mat R)\in\TRSO3 \qquad\text{for every }\mat R\in\SO3. \]
\end{definition}

By \Cref{eq:TangentSpace} the tangent bundle admits global left and right trivializations and can therefore be identified with $\SO3\times\mathfrak{so}(3)$.
Thus, $\TSO3$ is trivial.

More precisely, every smooth vector field $\vec F$ can be uniquely represented by smooth maps $X_L,X_R\colon\SO3\to\mathfrak{so}(3)$ such that
\[ \vec F(\mat R) = \mat R\,X_L(\mat R) = X_R(\mat R)\,\mat R, \qquad \mat R\in\SO3. \]
Under the corresponding trivializations, the vector field is represented by
\[ \mat R\mapsto\bigl(\mat R,X_L(\mat R)\bigr) \qquad\text{and}\qquad \mat R\mapsto\bigl(\mat R,X_R(\mat R)\bigr), \]
respectively.
Equivalently, $\SO3$ is parallelizable and admits global frames.

\begin{definition}
  Let $U\subset\SO3$ be open.
  A frame of the tangent bundle $\TSO3$ over $U$ is a triple of smooth vector fields
  \[ (\vec X_1,\vec X_2,\vec X_3), \qquad \vec X_i\colon U \to \TSO3, \]
  such that, for every $\mat R\in U$, the tangent vectors $\vec X_1(\mat R),\; \vec X_2(\mat R),\; \vec X_3(\mat R)$ form a basis of the tangent space $\TRSO3$.
  If $U=\SO3$, we call $\{\vec X_1,\vec X_2,\vec X_3\}$ a global frame.
\end{definition}

With respect to such a frame, every vector field $\vec F\colon U\to T\SO3$ can be written uniquely as
\[ \vec F(\mat R) = F_1(\mat R)\vec X_1(\mat R) + F_2(\mat R)\vec X_2(\mat R) + F_3(\mat R)\vec X_3(\mat R), \]
where \(F_1,F_2,F_3\colon U\to\IR\) are the corresponding scalar component functions.

In the following, we examine different frames for representing tangent vector fields on $\SO3$.

\subsection{The Euler Angle Induced Local Frame}\label{sec:TangentSpaceBasisEulerAngle}
At regular points of a local parametrization of a manifold, the columns of the Jacobian provide a basis of the corresponding tangent space.
Since both the Wigner-D functions and the Haar measure on $\SO3$ are naturally expressed in terms of Euler angles, the Euler angle parametrization is commonly used to construct a tangent space basis of $\TRSO3$, see~\cite{Bunge1986,Kalidindi2005,Graef2013}.
The resulting tangent vectors
can be identified with elements of the Lie algebra $\mathfrak{so}(3)$ via either left- or right-translation, i.e.
\begin{align*}
  \partial_{\alpha} \coloneqq \tfrac{\d}{\d\alpha}\epsilon\Big|_{\mat R} &=
  \mat R \cdot \left(\begin{smallmatrix}
		0 & -\cos\beta & \sin\beta\,\sin\gamma \\
		\cos\beta & 0 & \sin\beta\,\cos\gamma \\
    -\sin\beta\,\sin\gamma & -\sin\beta\,\cos\gamma & 0
  \end{smallmatrix}\right) &&=
  \left(\begin{smallmatrix}
		0 & -1 & 0 \\
 	1 & 0 & 0 \\
    0 & 0 & 0
  \end{smallmatrix}\right) \cdot \mat R, \\[2mm]
  \partial_{\beta} \coloneqq \tfrac{\d}{\d\beta}\epsilon\Big|_{\mat R} &=
  \mat R \cdot \left(\begin{smallmatrix}
		0 & 0 & \cos\gamma \\
 	0 & 0 & -\sin\gamma \\
    -\cos\gamma & \sin\gamma & 0
  \end{smallmatrix}\right) &&=
    \left(\begin{smallmatrix}
		0 & 0 & \cos\alpha \\
 	0 & 0 & \sin\alpha \\
    -\cos\alpha & -\sin\alpha & 0
  \end{smallmatrix}\right) \cdot \mat R, \\[2mm]
  \partial_{\gamma} \coloneqq \tfrac{\d}{\d\gamma}\epsilon\Big|_{\mat R} &=
  \mat R \cdot \left(\begin{smallmatrix}
		0 & -1 & 0 \\
    1 & 0 & 0 \\
    0 & 0 & 0 \\
  \end{smallmatrix}\right) &&=
  \left(\begin{smallmatrix}
		0 & -\cos\beta & \sin\alpha\,\sin\beta \\
		\cos\beta & 0 & -\cos\alpha\,\sin\beta \\
    -\sin\alpha\sin\beta & \cos\alpha\sin\beta & 0
  \end{smallmatrix}\right) \cdot \mat R .
\end{align*}

The tangent space basis depends on the Euler angles and is not orthogonal with respect to the Riemannian metric.
Moreover, as $\beta$ approaches $0$ or $\pi$, the loss of local injectivity of the Euler angle parametrization $\epsilon$ causes two of the basis vectors to become linearly dependent.
At $\beta=0$ and $\beta=\pi$, they therefore no longer form a basis of the corresponding tangent space.
Hence, this construction does not define a global frame of the tangent bundle, but only a local frame over
\[ \epsilon(\IT\times(0,\pi)\times\IT)\subset\SO3. \]
Consequently, this local tangent frame becomes numerically ill-conditioned near $\beta=0$ and $\beta=\pi$, making it inconvenient for frequency domain computations with tangent vector fields.

\subsection{Global Orthonormal Frames of $\TSO3$}\label{sec:TangentSpaceBasisStandard}

To construct a numerically more stable tangent space basis, we start from the standard basis
\begin{equation*}
  \ex \coloneqq
  \left(\begin{smallmatrix}
		0 & 0 & 0 \\
 	0 & 0 & -1 \\
    0 & 1 & 0
   \end{smallmatrix}\right) = \tfrac{\partial \mat R_{\vec x}(t)}{\partial t}\Big|_{t=0}, \qquad
 \ey \coloneqq
  \left(\begin{smallmatrix}
		0 & 0 & 1 \\
 	0 & 0 & 0 \\
    -1 & 0 & 0
   \end{smallmatrix}\right)= \tfrac{\partial \mat R_{\vec y}(t)}{\partial t}\Big|_{t=0}, \qquad
 \ez \coloneqq
  \left(\begin{smallmatrix}
		0 & -1 & 0 \\
 	1 & 0 & 0 \\
    0 & 0 & 0
   \end{smallmatrix}\right)= \tfrac{\partial \mat R_{\vec z}(t)}{\partial t}\Big|_{t=0}
\end{equation*}
of $\mathfrak{so}(3)$.
For every $\mat R\in\SO3$, left translation of these Lie algebra elements yields the tangent space basis
\begin{align*}
 \partial_{\vec x}^{L} \coloneqq \mat R \cdot \ex, \quad \partial_{\vec y}^{L} \coloneqq  \mat R \cdot \ey, \quad \partial_{\vec z}^{L} \coloneqq  \mat R \cdot \ez 
\end{align*}
of $\TRSO3$, whereas right translation yields
\begin{align*}
 \partial_{\vec x}^{R} \coloneqq \ex \cdot \mat R, \quad \partial_{\vec y}^{R} \coloneqq  \ey \cdot \mat R, \quad \partial_{\vec z}^{R} \coloneqq \ez \cdot \mat R. 
\end{align*}
Both tangent space bases are orthonormal with respect to the Riemannian metric and correspond to infinitesimal rotations around the coordinate axes.

For any tangent vector $\mat T\in\TRSO3$ with
\[ \mat T = l_{1} \partial_{\vec x}^{L} + l_{2} \partial_{\vec y}^{L} + l_{3} \partial_{\vec z}^{L} = r_{1} \partial_{\vec x}^{R} + r_{2} \partial_{\vec y}^{R} + r_{3} \partial_{\vec z}^{R} \]
the basis transform between the left- and right-translated tangent space bases is given by 
\begin{equation}\label{eq:TangentSpaceBasisChangeLeftRight}
  \left(\begin{smallmatrix} \vphantom{l}r_{1} \\ \vphantom{l}r_{2} \\ \vphantom{l}r_{3} \end{smallmatrix}\right) = \mat R \cdot \left(\begin{smallmatrix} l_{1} \\ l_{2} \\ l_{3} \end{smallmatrix}\right).
\end{equation}

Since the two tangent space bases depend smoothly on $\mat R$, they define two global orthonormal frames of the tangent bundle $\TSO3$.
More precisely, the frame given by $\partial_{\vec x}^{L}$, $\partial_{\vec y}^{L}$, and $\partial_{\vec z}^{L}$ is the standard left-invariant frame, whereas the frame given by $\partial_{\vec x}^{R}$, $\partial_{\vec y}^{R}$, and $\partial_{\vec z}^{R}$ is the standard right-invariant frame. In the literature, these frames are also referred to as the body-fixed and space-fixed frames, respectively; see, e.g.,~\cite{Breev2023}.

In the following, we show that these two global frames are not only numerically more stable than the local frame \(\set{\partial_{\alpha},\partial_{\beta},\partial_{\gamma}}{}\) introduced in \Cref{sec:TangentSpaceBasisEulerAngle}, but also lead to compact and convenient representations of the gradient of Wigner-D functions and, consequently, of band-limited harmonic series expansions on $\SO3$.


\section{Harmonic Expansion of Vector Fields on \texorpdfstring{$\SO3$}{SO(3)}}\label{sec:VectorFields}

In this section, we introduce the construction of harmonic vector fields on $\SO3$ and describe how fundamental differential operators such as the gradient, divergence, and curl can be expressed directly in terms of the harmonic coefficients, which simplifies their computation and accelerates numerical simulations.

\subsection{Global Frame Representations of Vector Fields on \texorpdfstring{$\SO3$}{SO(3)}}
Let $\vec F\colon\SO3\to\TSO3$ be a smooth vector field.
With respect to the tangent frames introduced in \Cref{sec:TangentBundleFrames}, it admits the component representations
\begin{align*}
  \vec F  &= F_{l,x}\,\partial^{L}_{\vec x} + F_{l,y}\,\partial^{L}_{\vec y} + F_{l,z}\,\partial^{L}_{\vec z}\\
          &= F_{r,x}\,\partial^{R}_{\vec x} + F_{r,y}\,\partial^{R}_{\vec y} + F_{r,z}\,\partial^{R}_{\vec z}\\
          &= F_{\alpha}\,\partial_{\alpha} + F_{\beta}\,\partial_{\beta} + F_{\gamma}\,\partial_{\gamma}.
\end{align*}
Thus, a smooth vector field $\vec F$ can equivalently be described by its component functions with respect to a chosen tangent bundle frame.
If these components lie in $\L2SO3$, they can be expanded in harmonic series. Employing Fourier techniques~\cite{Hielscher2026} then enables both efficient computation and straightforward evaluation.

For numerical computations it is important to know whether a change of tangent space representation preserves the finite-dimensional harmonic structure. The next result shows that passing between the two global orthonormal frames does not destroy band-limitedness: it increases the required harmonic bandwidth by at most one. Thus, vector fields represented in either basis can still be treated within a finite-dimensional Fourier framework.
\begin{lemma}\label{lem:FrameChangeBandWidth}
  Let $\vec F\colon\SO3\to\TSO3$ be a vector field with representations
  \begin{equation*}
    \vec F = F_{l,x}\,\partial^{L}_{\vec x} + F_{l,y}\,\partial^{L}_{\vec y} + F_{l,z}\,\partial^{L}_{\vec z}
           = F_{r,x}\,\partial^{R}_{\vec x} + F_{r,y}\,\partial^{R}_{\vec y} + F_{r,z}\,\partial^{R}_{\vec z}.
  \end{equation*}

  If the component functions $F_{l,x}, F_{l,y}, F_{l,z}$ are $L$-band-limited, then $F_{r,x}, F_{r,y}, F_{r,z}$ are $(L+1)$-band-limited.
  Conversely, if $F_{r,x}, F_{r,y}, F_{r,z}$ are $L$-band-limited, then $F_{l,x}, F_{l,y}, F_{l,z}$ are $(L+1)$-band-limited.
\end{lemma}
\begin{proof}
  The basis transform from \cref{eq:TangentSpaceBasisChangeLeftRight} implies
  \begin{equation}\label{eq:TangentFrameLeft2Right}
    \begin{pmatrix} F_{r,x}(\mat R) \\ F_{r,y}(\mat R) \\ F_{r,z}(\mat R) \end{pmatrix} = \mat R \cdot \begin{pmatrix} F_{l,x}(\mat R) \\ F_{l,y}(\mat R) \\ F_{l,z}(\mat R) \end{pmatrix}
  \end{equation}
  pointwise for all $\mat R \in \SO3$.

  Writing the rotation matrix $\mat R$ with respect to the Euler angle parametrization, each matrix entry is a linear combination of Wigner-D functions of degree $1$. Explicitly,
  \begin{align}\label{eq:RotationAsWignerD}
    \mat R = \begin{pmatrix}
      \frac{\WignerD{1}{-1}{-1} + \WignerD{1}{1}{-1} + \WignerD{1}{-1}{1} + \WignerD{1}{1}{1}}{\sqrt{12}} & \frac{-\i\WignerD{1}{-1}{-1} -\i\WignerD{1}{1}{-1} + \i\WignerD{1}{-1}{1} + \i\WignerD{1}{1}{1}}{\sqrt{12}} & \frac{\WignerD{1}{-1}{0} + \WignerD{1}{1}{0}}{\sqrt6}\\[0.5em]
      \frac{\i\WignerD{1}{-1}{-1} -\i\WignerD{1}{1}{-1} + \i\WignerD{1}{-1}{1} -\i\WignerD{1}{1}{1}}{\sqrt{12}} & \frac{\WignerD{1}{-1}{-1} -\WignerD{1}{1}{-1} - \WignerD{1}{-1}{1} + \WignerD{1}{1}{1}}{\sqrt{12}} & \frac{\i\WignerD{1}{-1}{0} - \i\WignerD{1}{1}{0}}{\sqrt6} \\[0.5em]
      \frac{\WignerD{1}{0}{-1} + \WignerD{1}{0}{1}}{\sqrt6} & \frac{-\i\WignerD{1}{0}{-1} + \i\WignerD{1}{0}{1}}{\sqrt6} & \frac{\WignerD{1}{0}{0}}{\sqrt3}
    \end{pmatrix}.
  \end{align}
  Using the assumption that the harmonic expansions of the component functions $F_{l,x}, F_{l,y}, F_{l,z}$ are $L$-band-limited, each of the functions $F_{r,x}, F_{r,y}, F_{r,z}$ can be written as linear combinations of products
  \[ \WignerD{l}{k}{k'}(\mat R)\,\WignerD{1}{j}{j'}(\mat R), \qquad l\leq L, \quad \abs{k},\abs{k'}\leq l, \quad j,j'\in\set{-1,0,1}{}. \]

  Since $\WignerD{l}{k}{k'}$ and $\WignerD{1}{j}{j'}$ are matrix elements of irreducible representations of $\SO3$, their pointwise product corresponds to the tensor product representation.
  By the Clebsch-Gordan decomposition, this product expands into a linear combination of Wigner-D functions of degrees $l-1$, $l$, and $l+1$, see \cite{Varshalovich1988}.

  Hence, the maximal harmonic degree appearing in the expansions of $F_{r,x}, F_{r,y}, F_{r,z}$ is $L+1$.

  The converse follows analogously, since $\mat R^{-1} = \mat R^\top$ and thus the same argument applies.

\end{proof}

Thus, the two representations have essentially the same harmonic bandwidth. The remaining question is how to compute the transformation between the left- and right-invariant frames in the frequency domain.

\begin{remark}
  Following the construction in the proof, the transformation between the two global orthonormal frames can also be carried out directly on the harmonic coefficients.
  More precisely, we have to insert the Wigner-D representation of the rotation matrix~\eqref{eq:RotationAsWignerD} into the tangent vector transform~\eqref{eq:TangentFrameLeft2Right} and substitute the harmonic series expansions of the component functions $F_{l,\cdot}$ and $F_{r,\cdot}$.
  Afterwards, the products of the Wigner-D functions can be written with the appropriate Clebsch-Gordan coefficients arising from the coupling
  \begin{equation*}
    \WignerD{l}{k}{k'}(\mat R)\, \WignerD{1}{j}{j'}(\mat R) = \sum_{J = l-1}^{l+1} C^{J,\,k+j}_{l,k;\,1,j} \, C^{J,\,k'+j'}_{l,k';\,1,j'} \, \WignerD{J}{k+j}{k'+j'}(\mat R).
  \end{equation*}
  This yields a complicated but computable formula for the change of tangent frame in the frequency domain.

  However, by the previous lemma, we know that the harmonic expansions of the component functions are $L+1$-band-limited after the change of tangent frame.
  Hence, we can compute them exactly by computing an adjoint $\SO3$-Fourier transform on the Clenshaw-Curtis quadrature grid, as described in detail in~\cite{Hielscher2026,Potts2009}.
  In practice, this is much faster than the explicit computation with the Clebsch-Gordon coefficients.
\end{remark}

In the following sections, we study gradients of functions on $\SO3$ as well as the divergence, curl, and antiderivative of vector fields, and express them in terms of harmonic series expansions.
In this context, the specific choice of tangent frame matters.
Assume that a function $f\colon\SO3\to\IC$ satisfies a certain symmetry with respect to left multiplication, i.e., $f(\mat Q\mat R)=f(\mat R)$ for all $\mat Q\in S\subset\SO3$.
Then the component functions of its gradient with respect to the left-invariant tangent frame inherit the same symmetry, whereas the component functions with respect to the right-invariant frame do not.
This is particularly important in crystallographic texture analysis, which motivates our investigations, since functions describing crystal-lattice-dependent properties are always invariant under left multiplication by the crystallographic symmetry group.

\subsection{The Gradient in the Harmonic Setting}\label{sec:Gradient}
Consider a smooth function $f\colon\SO3\to\IR$.
The gradient $\nabla f_{\mat R}\in\TRSO3$ at $\mat R\in\SO3$ is defined as the unique tangent vector satisfying
\begin{equation*}
  \RieMetric{\nabla f_{\mat R}}{\mat T} = \d{}_{\mat T} f(\mat R), \qquad \text{for all } \mat T\in\TRSO3,
\end{equation*}
where $\d{}_{\mat T} f(\mat R)$ denotes the derivative of $f$ at $\mat R$ along the tangent direction $\mat T$.
More concretely, let $\mat S \in \mathfrak{so}(3)$ with $\mat S = s_{1} \, \ex + s_{2} \, \ey + s_{3} \, \ez$, such that $\mat T = \mat S \cdot \mat R$. Then,
\[ \d{}_{\mat T}(f(\mat R)) = \frac{\d}{\d{t}} f\big(\exp(t \mat S)\,\mat R\big)\Bigg|_{t=0} \]
where the exponential map $\exp\colon\mathfrak{so}(3)\to\SO3$ is defined by
\[ \exp(\mat S) = \begin{cases} \id,  &\text{if } \mat S = 0, \\ \mat R_{\vec s} \braces*{\tfrac{\norm{S}_{F}}{\sqrt2}},  & \text{otherwise}. \end{cases} \]

In the following, we will represent the gradient of a smooth function $f\colon\SO3\to\IR$ as a harmonic vector field $\grad f\colon\SO3\to\TSO3$, whose representation depends on the chosen tangent bundle frame.
Since $\SO3$ is compact, the smoothness of $f$ implies that it belongs to every Sobolev space on $\SO3$ and, in particular, to $\L2SO3$.
Hence, $f$ admits a harmonic expansion.

The following theorem is the key computational result of this section. It gives the gradient of a Wigner-D function with respect to the Euler angle induced local frame and the two global orthonormal frames introduced above. The main advantage of the two global frames is that the corresponding formulas involve only shifts of the harmonic indices and preserve the harmonic degree. In contrast, the representation with respect to the Euler angle induced frame contains explicit Euler angle dependent coefficients and singular factors.

For complex-valued functions, the gradient operator is extended complex-linearly.
\begin{theorem}\label{thm:Gradient}
  The gradient of the Wigner-D function $\WignerD{l}{k}{k'}\colon \SO3 \to \IC$, expressed with respect to the three tangent bundle frames introduced above, is given by
  \renewcommand{\arraystretch}{1.2}
  \begin{align*}
    \grad \WignerD{l}{k}{k'} &=
    \begin{pmatrix*}[c]
      \WignerD{l}{k-1}{k'} \\ \WignerD{l}{k}{k'} \\ \WignerD{l}{k+1}{k'}
    \end{pmatrix*}^{\top}
    \begin{pmatrix*}[c]
      0 & -c_{l}^{-k}\,\e^{-\i\alpha} & 0 \\
      \tfrac{\i\,k'\cos\beta-\i\,k}{2\sin^{2}\beta} & 0 & \tfrac{\i\,k\cos\beta-\i\,k'}{2\sin^{2}\beta} \\
      0 & c_{l}^{k}\,\e^{\i\alpha} & 0
    \end{pmatrix*}
    \begin{pmatrix*}[c]
      \partial_{\alpha} \\ \partial_{\beta} \\ \partial_{\gamma}
    \end{pmatrix*}
    \\[0.5em] &=
    \begin{pmatrix*}[c]
      \WignerD{l}{k-1}{k'} \\ \WignerD{l}{k}{k'} \\ \WignerD{l}{k+1}{k'}
    \end{pmatrix*}^{\top}
    \begin{pmatrix*}[c]
      -\i\,c_{l}^{-k'} & c_{l}^{-k'} & 0 \\
      0 & 0 & -\i\,k' \\
      -\i\,c_{l}^{k'} & -c_{l}^{k'} & 0
    \end{pmatrix*}
    \begin{pmatrix*}[c]
      \partial^{L}_{\vec x} \\ \partial^{L}_{\vec y} \\ \partial^{L}_{\vec z}
    \end{pmatrix*}
    \\[0.5em] &=
    \begin{pmatrix*}[c]
      \WignerD{l}{k}{k'-1} \\ \WignerD{l}{k}{k'} \\ \WignerD{l}{k}{k'+1}
    \end{pmatrix*}^{\top}
    \begin{pmatrix*}[c]
      -\i\,c_{l}^{-k} & -c_{l}^{-k} & 0 \\
      0 & 0 & -\i\,k \\
      -\i\,c_{l}^{k} & c_{l}^{k}  & 0
    \end{pmatrix*}
    \begin{pmatrix*}[c]
      \partial^{R}_{\vec x} \\ \partial^{R}_{\vec y} \\ \partial^{R}_{\vec z}
    \end{pmatrix*},
  \end{align*}
  \renewcommand{\arraystretch}{1.0}
  where $\WignerD{l}{k}{k'} = 0$ for $\abs{k},\abs{k'}>l$ and $c_{l}^{k} = \tfrac12\,\sqrt{(l-k)(l+k+1)}$.
\end{theorem}
\begin{proof}
  The gradient with respect to the Euler angle induced frame $\set{\partial_{\alpha},\partial_{\beta},\partial_{\gamma}}{}$ is stated in~\cite{Varshalovich1988}. More precisely, the explicit Euler angle derivatives of $\WignerD{l}{k}{k'}$ are inserted into the natural coordinate representation of the Riemannian gradient, i.e.,
  \begin{equation*}
    \grad \WignerD{l}{k}{k'} =
    \partial_{\alpha}\,\tfrac1{2\sin^{2}\beta}\left(\tfrac{\partial \WignerD{l}{k}{k'}}{ \partial \alpha}-\cos\beta\, \tfrac{\partial \WignerD{l}{k}{k'}}{ \partial \gamma}\right)
    + \partial_{\beta}\, \tfrac12 \tfrac{\partial \WignerD{l}{k}{k'}}{\partial \beta}
    + \partial_{\gamma}\, \tfrac1{2\sin^{2}\beta}\left(\tfrac{\partial \WignerD{l}{k}{k'}}{ \partial \gamma}-\cos\beta\, \tfrac{\partial \WignerD{l}{k}{k'}}{ \partial \alpha}\right).
  \end{equation*}

  Since the frame $\set{\partial^{R}{\vec x},\partial^{R}{\vec y},\partial^{R}{\vec z}}{}$ is orthonormal, the components of the gradient are given by the corresponding directional derivatives.
  Hence, by definition of the frame, we obtain
  \begin{equation*}
    \nabla \WignerD{l}{k}{k'} (\mat R) =
    \braces*{\frac{\d}{\d{t}}\WignerD{l}{k}{k'}(\mat R_{\vec x}(t)\cdot \mat R)\Bigg|_{t=0}} \, \partial^{R}_{\vec x} +
    \braces*{\frac{\d}{\d{t}}\WignerD{l}{k}{k'}(\mat R_{\vec y}(t)\cdot \mat R)\Bigg|_{t=0}} \, \partial^{R}_{\vec y} +
    \braces*{\frac{\d}{\d{t}}\WignerD{l}{k}{k'}(\mat R_{\vec z}(t)\cdot \mat R)\Bigg|_{t=0}} \, \partial^{R}_{\vec z}.
  \end{equation*}
  For $\vec \eta\in\S2$, the representation property~\eqref{eq:RepresentationProperty} yields
  \begin{equation*}
    \frac{\d}{\d{t}} \WignerD{l}{k}{k'}\big(\mat R_{\vec \eta}(t)\cdot\mat R\big)\Big|_{t=0}
    = \frac{1}{\sqrt{2l+1}}\,\sum_{j=-l}^l  \frac{\d}{\d{t}} \WignerD{l}{k}{j}\big(\mat R_{\vec \eta}(t)\big)\Big|_{t=0} \cdot \WignerD{l}{j}{k'}(\mat R).
  \end{equation*}
  \begin{enumerate}
    \item For $\vec \eta = \vec z$ we have $\mat R_{\vec z}(t) = \mat R(t,0,0)$, and hence
          \[ \WignerD{l}{k}{j}(\mat R_{\vec z}(t)) = \sqrt{2l+1}\,\e^{-\i kt}\,\Wignerd{l}{k}{j}(1) = \sqrt{2l+1}\,\e^{-\i kt}\, \charFunc{k=j} \]
          where $\charFunc{\cdot}$ denotes the characteristic function.
          Differentiating with respect to $t$ then yields
          \[ \frac{\d}{\d{t}} \WignerD{l}{k}{j}\big(\mat R_{\vec z}(t)\big)\Big|_{t=0} = -\i\,k\, \sqrt{2l+1} \, \charFunc{k=j}. \]
    \item For $\vec \eta = \vec y$, we have $\mat R_{\vec y}(t) = \mat R(0,t,0)$, which implies
          \[ \WignerD{l}{k}{j}(\mat R_{\vec y}(t)) = \sqrt{2l+1}\,\Wignerd{l}{k}{j}(\cos t). \]
          Differentiating with respect to \(t\) and using the derivative identity from \Cref{lem:DerivativeWignerd} yields
          \[ \frac{\d}{\d{t}} \WignerD{l}{k}{j}\big(\mat R_{\vec y}(t)\big)\Big|_{t=0} = c_l^k \, \sqrt{2l+1} \cdot \charFunc{k+1=j} - c_l^{-k} \, \sqrt{2l+1} \cdot \charFunc{k-1=j}. \]
    \item For $\vec \eta = \vec x$, we have $\mat R_{\vec x}(t) = \mat R(\frac32\pi,t,\frac{\pi}2)$, which implies
          \[ \WignerD{l}{k}{j}(\mat R_{\vec x}(t)) = \sqrt{2l+1}\, \i^k \, (-\i)^j \,\Wignerd{l}{k}{j}(\cos t). \]
          Differentiating with respect to \(t\) and using the derivative identity from \Cref{lem:DerivativeWignerd} yields
          \[ \frac{\d}{\d{t}} \WignerD{l}{k}{j}\big(\mat R_{\vec x}(t)\big)\Big|_{t=0} = -\i\, c_l^k \, \sqrt{2l+1} \cdot \charFunc{k+1=j} - \i\, c_l^{-k} \, \sqrt{2l+1} \cdot \charFunc{k-1=j}. \]
  \end{enumerate}

  The representation with respect to the frame $\set{\partial^{L}_{\vec x},\partial^{L}_{\vec y},\partial^{L}_{\vec z}}{}$ follows analogously.
\end{proof}

By linearity, the identities in \Cref{thm:Gradient} extend directly from individual Wigner-D functions to arbitrary harmonic expansions. This yields explicit coefficient formulas for the components of the gradient.

\begin{corollary}\label{cor:GradientHarmonicSeries}
  Let $f\colon\SO3\to\IR$ be a smooth function with harmonic coefficients $\big(\fhat{l}{k}{k'}\big)_{(l,k,k')\in\J{\infty}}$. 
  Then, the gradient of $f$ is given by
  \renewcommand{\arraystretch}{1.2}
  \begin{align*}
    \nabla f (\mat R) &=
    \begin{pmatrix}
      \ds\sum_{l=0}^{\infty}\sum_{k,k'=-l}^l \left( \tfrac{\i\,k'\cos\beta-\i\,k}{2\sin^{2}\beta}\fhat{l}{k}{k'} \right) \, \WignerD{l}{k}{k'}(\mat R) \\
      \ds\sum_{l=0}^{\infty}\sum_{k,k'=-l}^l \left( c_l^{-k} \, \e^{\i\alpha} \, \fhat{l}{k-1}{k'} - c_l^k \, \e^{-\i\alpha} \, \fhat{l}{k+1}{k'}\right) \, \WignerD{l}{k}{k'}(\mat R) \\
      \ds\sum_{l=0}^{\infty}\sum_{k,k'=-l}^l \left( \tfrac{\i\,k\cos\beta-\i\,k'}{2\sin^{2}\beta}\fhat{l}{k}{k'} \right) \, \WignerD{l}{k}{k'}(\mat R)
    \end{pmatrix}^{\top}
    \begin{pmatrix*}[c]
      \partial_{\alpha} \\ \partial_{\beta} \\ \partial_{\gamma}
    \end{pmatrix*}
    \\[0.5em] &=
    \begin{pmatrix}
      \ds\sum_{l=0}^{\infty}\sum_{k,k'=-l}^l \left(-\i \, c_l^{-k'} \, \fhat{l}{k}{k'-1} - \i \, c_l^{k'} \, \fhat{l}{k}{k'+1}\right) \, \WignerD{l}{k}{k'}(\mat R) \\
      \ds\sum_{l=0}^{\infty}\sum_{k,k'=-l}^l \left(-c_l^{-k'} \, \fhat{l}{k}{k'-1} + c_l^{k'} \, \fhat{l}{k}{k'+1}\right) \, \WignerD{l}{k}{k'}(\mat R) \\
      \ds\sum_{l=0}^{\infty}\sum_{k,k'=-l}^l \left(-\i\,k'\,\fhat{l}{k}{k'}\right) \, \WignerD{l}{k}{k'}(\mat R)
    \end{pmatrix}^{\top}
    \begin{pmatrix*}[c]
      \partial^{L}_{\vec x} \\ \partial^{L}_{\vec y} \\ \partial^{L}_{\vec z}
    \end{pmatrix*}
    \\[0.5em] &=
    \begin{pmatrix}
      \ds\sum_{l=0}^{\infty}\sum_{k,k'=-l}^l \left(-\i \, c_l^{-k} \, \fhat{l}{k-1}{k'} - \i \, c_l^k \, \fhat{l}{k+1}{k'}\right) \, \WignerD{l}{k}{k'}(\mat R) \\
      \ds\sum_{l=0}^{\infty}\sum_{k,k'=-l}^l \left(c_l^{-k} \, \fhat{l}{k-1}{k'} - c_l^k \, \fhat{l}{k+1}{k'}\right) \, \WignerD{l}{k}{k'}(\mat R) \\
      \ds\sum_{l=0}^{\infty}\sum_{k,k'=-l}^l \left(-\i\,k\,\fhat{l}{k}{k'}\right) \, \WignerD{l}{k}{k'}(\mat R)
    \end{pmatrix}^{\top}
    \begin{pmatrix*}[c]
      \partial^{R}_{\vec x} \\ \partial^{R}_{\vec y} \\ \partial^{R}_{\vec z},
    \end{pmatrix*}
  \end{align*}
  \renewcommand{\arraystretch}{1.0}
  where $\WignerD{l}{k}{k'}(\mat R) = 0$ for $\abs{k},\abs{k'}>l$ and $c_{l}^{k} = \tfrac12\,\sqrt{(l-k)(l+k+1)}$.
\end{corollary}

The coefficient formulas in \cref{cor:GradientHarmonicSeries} show in particular that, with respect to the two global orthonormal tangent bundle frames, differentiation does not increase the harmonic degree. This gives the following bandwidth-preservation property.

\begin{corollary}\label{cor:GradientBandLimited}
  With respect to the frames $\set{\partial^{L}_{\vec x},\partial^{L}_{\vec y},\partial^{L}_{\vec z}}{}$ and $\set{\partial^{R}_{\vec x},\partial^{R}_{\vec y},\partial^{R}_{\vec z}}{}$, the components of the gradient of an $L$-band-limited function $f\colon\SO3\to\IC$ are again $L$-band-limited.
\end{corollary}
Moreover, the harmonic coefficients of the gradient components can be obtained directly in Fourier space from the harmonic coefficients of $f$.

This bandwidth preservation is the main computational advantage of the two global orthonormal frames. If a rotation-dependent function is represented by a finite harmonic expansion, its gradient can be computed directly from its harmonic coefficients without increasing the prescribed bandwidth. Thus, no numerical differentiation, no treatment of Euler angle singularities, and no additional projection step are required. This property is used in~\cref{sec:Compactification} for the gradient-based construction of representative rotations and later for the harmonic construction of the anisotropic rotational Smoluchowski operator.

When expressed in the local frame $\set{\partial_{\alpha},\partial_{\beta},\partial_{\gamma}}{}$, the harmonic coefficients of the gradient components are not readily accessible, since the component formulas involve explicit Euler angle factors. In principle, they can be computed by expanding the components into harmonic series using numerical quadrature in combination with a fast adjoint $\SO3$ Fourier transform, see~\cite{Hielscher2026,Potts2009}. This approach, however, does not preserve band-limitedness and therefore leads to unavoidable truncation errors in practical computations.

\subsection{Differential operators on \texorpdfstring{$\SO3$}{SO(3)}}\label{sec:DifferentialOperators}

By construction, the Wigner-D functions are the eigenfunctions of the Laplace-Beltrami operator, i.e.
\[ \laplaceB \WignerD{l}{k}{k'} = -l(l+1) \WignerD{l}{k}{k'}. \]
Consequently, for any smooth function $f$ with harmonic coefficients $\big(\fhat{l}{k}{k'}\big)_{(l,k,k')\in\J{\infty}}$, the Laplace-Beltrami operator acts diagonally on the harmonic expansion, i.e.
\[ \laplaceB f = \sum_{l=0}^{\infty}\sum_{k,k'=-l}^{l} (-l(l+1))\fhat{l}{k}{k'} \, \WignerD{l}{k}{k'}. \]

Following~\cite{Varshalovich1988}, we obtain the following formulas for the divergence and curl of smooth vector fields represented in the left- or right-invariant frame.
\begin{lemma}\label{lem:DivergenceCurl}
  Let $\vec F\colon\SO3\to\TSO3$ be a smooth vector field with representations
  \begin{equation*}
    \vec F = F_{l,x}\,\partial^{L}_{\vec x} + F_{l,y}\,\partial^{L}_{\vec y} + F_{l,z}\,\partial^{L}_{\vec z} = F_{r,x}\,\partial^{R}_{\vec x} + F_{r,y}\,\partial^{R}_{\vec y} + F_{r,z}\,\partial^{R}_{\vec z}.
  \end{equation*}
  Then the divergence of $\vec F$ is given by
  \begin{align*}
    \div\vec F &= \braces[\big]{\grad F_{l,x}}_{l,x} + \braces[\big]{\grad F_{l,y}}_{l,y} + \braces[\big]{\grad F_{l,z}}_{l,z} \\
               &= \braces[\big]{\grad F_{r,x}}_{r,x} + \braces[\big]{\grad F_{r,y}}_{r,y} + \braces[\big]{\grad F_{r,z}}_{r,z},
  \end{align*}
  and its curl is given by
  \begin{align*}
    \curl\vec F &=
    \begin{pmatrix*}
      \braces[\big]{\grad F_{l,y}}_{l,z} - \braces[\big]{\grad F_{l,z}}_{l,y} + F_{l,x} \\
      \braces[\big]{\grad F_{l,z}}_{l,x} - \braces[\big]{\grad F_{l,x}}_{l,z} + F_{l,y} \\
      \braces[\big]{\grad F_{l,x}}_{l,y} - \braces[\big]{\grad F_{l,y}}_{l,x} + F_{l,z}
    \end{pmatrix*}^{\top}
    \begin{pmatrix*}
      \partial^{L}_{\vec x} \\ \partial^{L}_{\vec y} \\ \partial^{L}_{\vec z}
    \end{pmatrix*}
    \\
    &=
    \begin{pmatrix*}
      \braces[\big]{\grad F_{r,y}}_{r,z} - \braces[\big]{\grad F_{r,z}}_{r,y} - F_{r,x} \\
      \braces[\big]{\grad F_{r,z}}_{r,x} - \braces[\big]{\grad F_{r,x}}_{r,z} - F_{r,y} \\
      \braces[\big]{\grad F_{r,x}}_{r,y} - \braces[\big]{\grad F_{r,y}}_{r,x} - F_{r,z}
    \end{pmatrix*}^{\top}
    \begin{pmatrix*}
      \partial^{R}_{\vec x} \\ \partial^{R}_{\vec y} \\ \partial^{R}_{\vec z}
    \end{pmatrix*}.
  \end{align*}
  Here, $\braces[\big]{\grad F_{l,x}}_{l,x}$ denotes the $x$-component of the gradient of the scalar component function $F_{l,x}$ with respect to the left-invariant frame.
\end{lemma}

Combining this lemma with the harmonic representation of the gradient derived in \Cref{cor:GradientHarmonicSeries} immediately yields corresponding frequency domain representations of the divergence and curl of a smooth vector field.
Hence, both operators can be computed directly from the harmonic coefficients of the component functions.
Moreover, if the component functions of $\vec F$ are $L$-band-limited, then $\div\vec F$ and the component functions of $\curl\vec F$ are also $L$-band-limited.

Furthermore, for gradient vector fields, the underlying scalar potential can be reconstructed, up to an additive constant, directly from the frequency domain representation of the gradient.


\section{Applications and Numerical Experiments}

The harmonic framework for vector fields on $\SO3$ developed in the previous sections is a central component of \texttt{MTEX}~7.0~\cite{MTEX}.
In this section, we illustrate its numerical advantages through two examples.
The first, presented in \Cref{sec:Compactification}, concerns the construction of small representative rotation sets for rotation distribution functions and demonstrates the efficient computation of gradients of functions on $\SO3$ directly from their harmonic coefficients.
The second example, presented in \Cref{sec:BrownianMotion}, demonstrates how the harmonic representations of vector fields and differential operators can be combined as modular building blocks for the numerical solution of evolution equations on $\SO3$, using anisotropic rotational diffusion as a representative example.

\subsection{Optimal Sampling on $\SO3$}\label{sec:Compactification}

Let $f\colon \SO3 \to \IR$ be an density function of bandwidth $L$ on the rotation group, with harmonic coefficients $\big(\fhat{l}{k}{k'}\big)_{(l,k,k')\in\J{L}}$, i.e., $f \geq 0$ and $\int_{\SO3} f(\mat q)\d{\mu(\mat q)} = 8\pi^{2}$.

Given a prescribed number $M\in\IN$, the optimal sampling problem consists of finding a set of rotations $\mathcal R_M = \set{\mat R_1,\dots,\mat R_M}{}\subset \SO3$ whose associated discrete measure $8\pi^2\sum_{m=1}^M \omega_{m}\delta_{\mat R_m}$
approximates the measure $\d\nu = f(\mat q)\d{\mu(\mat q)}$ as accurately as possible.
Such reduced rotation sets are used, for example, as compact representations of rotation density functions in crystal-plasticity simulations~\cite{Eghtesad2018,Knezevic2014,Knezevic2016} and for the reduction of large rotation data sets, see, e.g.,~\cite{Wright1990,Baudin1993,Baudin1995,Pospiech1994}.

Different formulations of this problem have been proposed by Knezevic et al.~\cite{Knezevic2015,Knezevic2024} and by Gräf et al.~\cite{Graef2012,Graef2013,Dick2023}. The former optimize both the rotations and their weights by matching harmonic coefficients, whereas the latter consider equally weighted point sets and minimize a worst-case quadrature error in a reproducing kernel Hilbert space. Both approaches lead to similar optimization problems on $\SO3$, which are solved by conjugate-gradient methods using the gradient of the corresponding objective functional.
In the existing formulations, the gradient required for this optimization is expressed with respect to the tangent bundle frame induced by the Euler angle parametrization, as detailed in \Cref{sec:TangentSpaceBasisEulerAngle}.

In this paper, we follow the optimal sampling formulation of Gräf et al., but compute the gradient with respect to the global orthonormal frame introduced in \Cref{sec:TangentSpaceBasisStandard} and solve the resulting optimization problem by steepest descent.
With respect to this frame, the harmonic coefficients of the gradient components are obtained directly from those of the underlying scalar objective function, without requiring numerical differentiation or increasing the harmonic bandwidth, cf. \Cref{cor:GradientHarmonicSeries} and \ref{cor:GradientBandLimited}.

More precisely, following~\cite{Graef2012,Graef2013}, we consider the objective functional $\mathcal E\colon \SO3^M\to\IR$, given by
\begin{equation}\label{eq:MiniFunctionalE}
  \mathcal{E}(\mathcal R_{M}) = \sum_{l=0}^{L}\sum_{k,k'=-l}^{l} \lambda_{l} \abs*{ \frac{1}{M} \sum_{m=1}^{M} \conj{\WignerD{l}{k}{k'}(\mat R_{m})} - \fhat{l}{k}{k'} }^{2},
\end{equation}
where the coefficients $\lambda_{l}$ specify the discrepancy kernel used to measure the approximation quality of the point set $\mathcal R_{M}$.
More precisely, the associated discrepancy measures, in an averaged sense, how well the discrete measure reproduces the mass of metric balls in $\SO3$, taken over all centers and radii.

The gradient of $\mathcal{E}$ with respect to the $m$-th rotation can be written as
\begin{equation*}
  \nabla^{\mat R_{m}} \mathcal{E}(\mathcal R_{M})  = \frac2M \mathrm{Re}\bigl(\nabla C(\mat R_{m})\bigr),
\end{equation*}
where the scalar band-limited function $C\colon\SO3\to\mathbb C$ is given by
\begin{equation*}
  C(\mat R) = \sum_{l=0}^{L}\sum_{k,k'=-l}^{l} \lambda_{l} \left(\frac1M\sum_{j=1}^{M}\conj{\WignerD{l}{k}{k'}(\mat R_{j})} - \fhat{n}{k}{l}\right)   \WignerD{l}{k}{k'}(\mat R).   
\end{equation*}
The harmonic coefficients of $C$ are obtained from the current point set $\mathcal{R}_{M}$ by an adjoint $\SO3$-Fourier transform~\cite{Potts2009,Kostelec2008,Risbo1996,Hielscher2026} and multiplication with the kernel coefficients $\lambda_{l}$. Consequently, each gradient-descent step consists of forming the harmonic coefficients of $C$, computing the harmonic coefficients of the components of $\nabla C$ by \Cref{cor:GradientHarmonicSeries}, and evaluating the resulting vector field at the current rotation set $\mathcal R_{M}$ through a $\SO3$-Fourier transform.

Afterwards, the rotations are updated in the negative Riemannian gradient according to
\[ \mat R_{m}^{\text{new}} = \exp_{\mat R_{m}} \braces*{ - s \cdot \nabla^{\mat R_{m}}\mathcal{E}(\mathcal R_{M})} , \qquad \text{for } m=1,\dots,M, \]
where the step size $s>0$ is determined using an Armijo line search.

As a numerical example, we approximate the density shown in \Cref{fig:ODF3d} by $M=1\,024$ equally weighted representative rotations, minimizing the discrepancy $\mathcal{E}$ with bandwidth $L=32$.
Starting from an initial point set randomly sampled from the density function, cf.~\cite{Hielscher2013}, we perform $10\,000$ gradient-descent steps using the harmonic gradient computation described above.
The resulting optimized rotation set is shown in \Cref{fig:Compactification3d}.

To assess the quality of the point set, we reconstruct an density $\tilde f$ by kernel density estimation, cf.~\cite{Hielscher2013,Hielscher2010,Silverman2018,Schaeben2016}, using a de la Vall\'{e}e Poussin kernel with halfwidth $5.3^{\circ}$.
The reconstructed density is shown in \Cref{fig:DensityEstimation3d}.
The optimized point set yields the relative reconstruction error
\[ \frac{\norm{\tilde f-f}_{\L2SO3}}{\norm{f}_{\L2SO3}} \approx 0.057. \]
For comparison, a reconstruction based on $1\,024$ rotations randomly sampled according to the density function yields a relative error of $0.152$, even when the kernel halfwidth is chosen to minimize the reconstruction error.

To illustrate the reconstruction performance of the optimal sampling procedure, we study in \Cref{fig:SamplingError} the relative reconstruction error for the density function considered above as a function of the number of sample points.

\begin{figure}[H]
  \centering
  \begin{subfigure}[t]{0.32\textwidth}
    \centering
    \includegraphics[width = 0.98\textwidth]{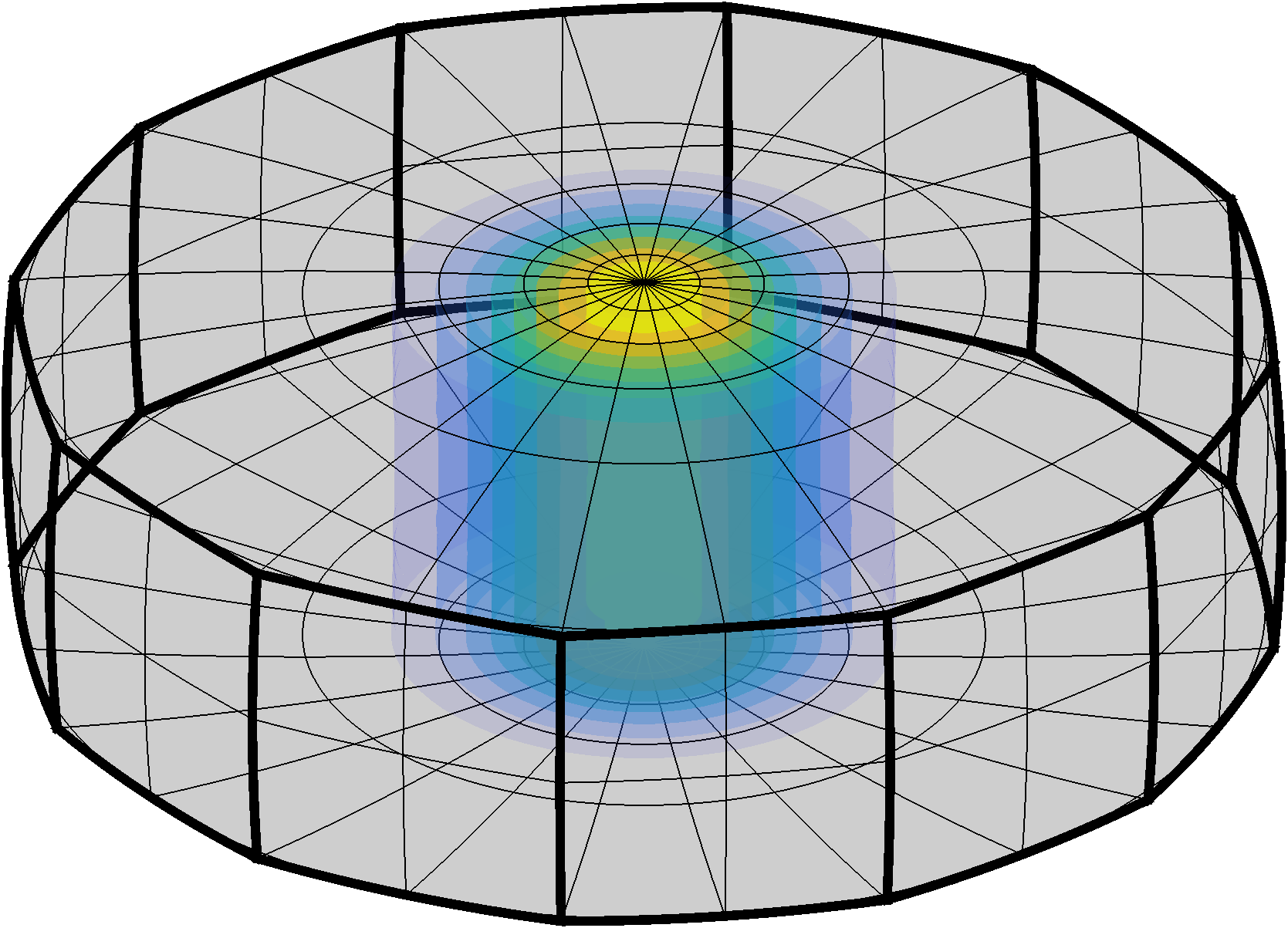}
    \subcaption{Target density}\label{fig:ODF3d}
  \end{subfigure}
  \begin{subfigure}[t]{0.32\textwidth}
    \centering
    \includegraphics[width = 0.98\textwidth]{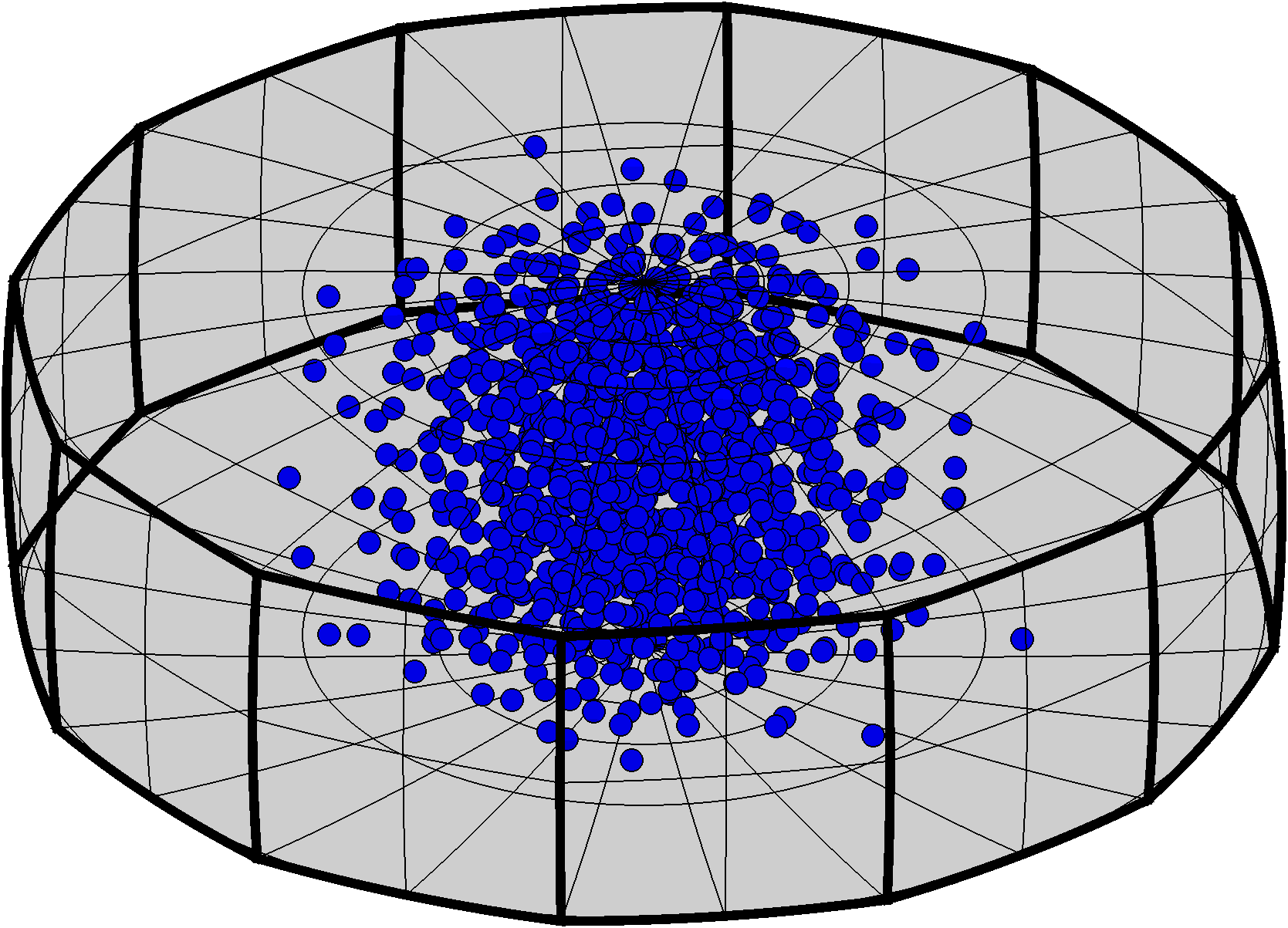}
    \subcaption{Optimized point set}\label{fig:Compactification3d}
  \end{subfigure}
  \begin{subfigure}[t]{0.32\textwidth}
    \centering
    \includegraphics[width = 0.98\textwidth]{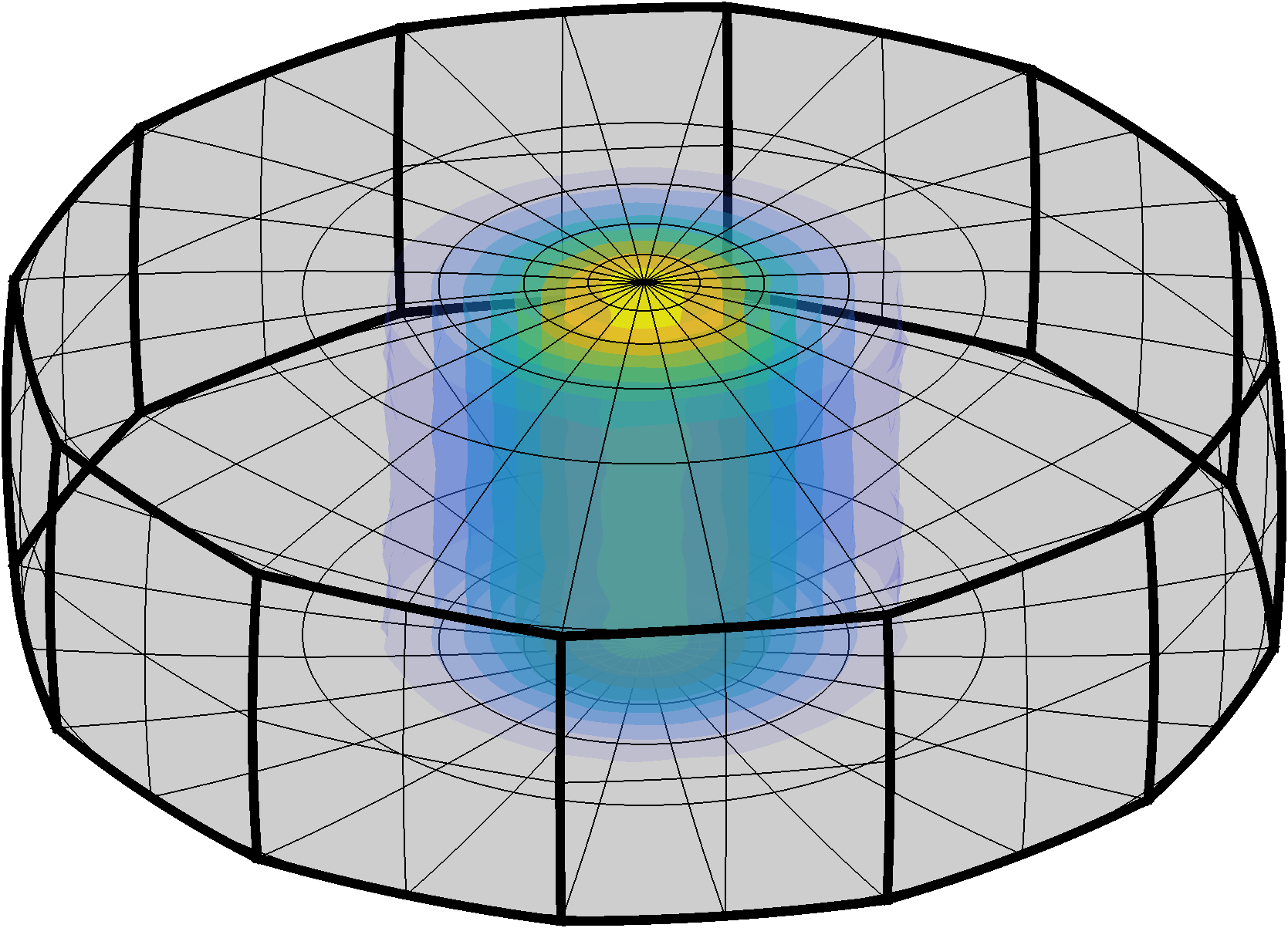}
    \subcaption{Reconstructed density}\label{fig:DensityEstimation3d}
  \end{subfigure}
  \caption{Visualization of the optimal-sampling experiment in the three-dimensional axis-angle representation of rotation space. Only the fundamental sector corresponding to the $622$ crystal symmetry is shown.}\label{fig:CompactificationODF3d}
\end{figure}

\definecolor{optred}{RGB}{200,55,55}
\definecolor{randblue}{RGB}{20,80,210}
\definecolor{boxblue}{RGB}{30,90,200}

\pgfplotstableread{pictures/OptimalSampling/SamplingError.txt}\samplingdata
\begin{figure}[H]
  \centering
  \begin{tikzpicture}
    \begin{axis}[
      width=13cm,
      height=8cm,
      xlabel={Number of rotations},
      ylabel={Relative error},
      boxplot/draw direction=y,
      xmin=2.5,
      xmax=11.5,
      ymin=0,
      ymax=0.6,
      ytick={0,0.1,...,0.8},
      xtick={3,4,5,6,7,8,9,10,11},
      xticklabels={32,64,128,256,512,1024,2048,4096,8192},
      grid=major,
      legend pos=north east,
      boxplot/every box/.style={draw=boxblue,fill=boxblue,fill opacity=0.15,thin},
      boxplot/every median/.style={draw=boxblue,semithick},
      boxplot/every whisker/.style={draw=black,semithick}
      ]

      \pgfplotsinvokeforeach{2,...,10}{
        \pgfmathtruncatemacro{\xpos}{#1+1}
        \pgfplotstablegetelem{#1}{Q1}\of\samplingdata
        \edef\qone{\pgfplotsretval}
        \pgfplotstablegetelem{#1}{Median}\of\samplingdata
        \edef\med{\pgfplotsretval}
        \pgfplotstablegetelem{#1}{Q3}\of\samplingdata
        \edef\qthree{\pgfplotsretval}
        \pgfplotstablegetelem{#1}{LowerWhisker}\of\samplingdata
        \edef\lowerwhisker{\pgfplotsretval}
        \pgfplotstablegetelem{#1}{UpperWhisker}\of\samplingdata
        \edef\upperwhisker{\pgfplotsretval}
        \edef\boxplotcommand{
          \noexpand\addplot[
          boxplot prepared={
            draw position=\xpos,
            lower whisker=\lowerwhisker,
            lower quartile=\qone,
            median=\med,
            upper quartile=\qthree,
            upper whisker=\upperwhisker,
            box extend=0.5
          },
          forget plot
          ]
          coordinates {};
        }
        \boxplotcommand
      }

      \addplot[optred,thick,mark=*,mark size=2.2pt,mark options={solid,draw=optred,fill=optred},restrict x to domain=3:11]
      table[x expr={\coordindex+1},y=ErelOpt]{\samplingdata};

      \addplot[randblue,thick,dashed,mark=triangle*,mark size=2.5pt,mark options={solid,draw=randblue,fill=randblue},restrict x to domain=3:11]
      table[x expr={\coordindex+1},y=MeanRand]{\samplingdata};

      \legend{Optimal sampling,Random sampling}
    \end{axis}
  \end{tikzpicture}
  \caption{Relative reconstruction error as a function of the number of rotations for optimal sampling and random sampling according to the density function.
    For each point set, the kernel halfwidth is chosen to minimize the reconstruction error.
    For random sampling, the boxplots show the distribution of the reconstruction errors over $100$ randomly generated point sets, while the dashed curve indicates the corresponding mean error.}\label{fig:SamplingError}
\end{figure}
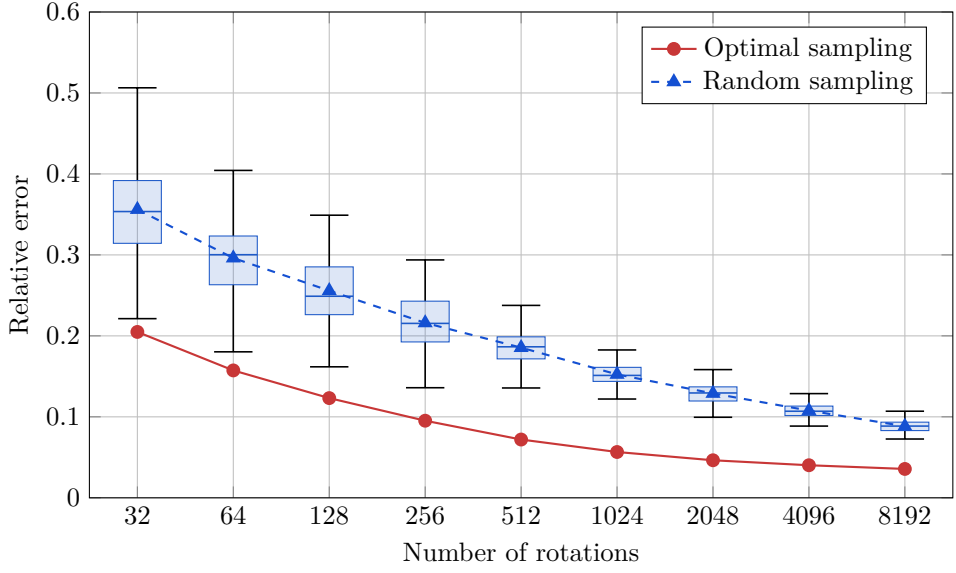

\subsection{Anisotropic Rotational Diffusion}\label{sec:BrownianMotion}

As a second application, we demonstrate how the harmonic vector-field framework developed above provides a modular approach to the numerical simulation of evolution equations on $\SO3$.
The frequency domain representations of the gradient and divergence, together with algebraic operations on tangent vector fields, provide the basic building blocks for a broad class of drift, diffusion, and transport equations.
Such equations arise, for example, in rotational Brownian motion~\cite{Favro1960,Coffey2011,Tarroni1991} and texture evolution~\cite{Morawiec2004,Bunge1984,Bunge1986,Kalidindi2005}.
Here, we illustrate this approach for anisotropic rotational diffusion.

We consider the rotational Brownian motion of an asymmetric rigid molecule in an orienting medium~\cite{Tarroni1991}.
Its orientation is described by a rotation $\mat R \in \SO3$, and its time-dependent rotation probability density is represented by
\[ f \colon \IR_{+} \times \SO3 \to \IR.\]
The dynamics combines anisotropic rotational diffusion with a systematic drift induced by an rotation-dependent potential $U \colon \SO3 \to \mathbb{R}$.

Following the model studied by Tarroni and Zannoni~\cite{Tarroni1991}, the rotational diffusion is characterized by three diffusion coefficients $D_x$, $D_y$, and $D_z$ associated with the molecule-fixed principal axes.
Accordingly, we define the diffusion tensor $\mathcal D$ of a vector field $\vec F = F_x \partial_{\vec x}^{L} + F_y \partial_{\vec y}^{L} + F_z \partial_{\vec z}^{L}$ with respect to the left-invariant tangent frame by
\[ \mathcal D \vec F = D_x F_x \partial_{\vec x}^{L} + D_y F_y \partial_{\vec y}^{L} + D_z F_z \partial_{\vec z}^{L}. \]
The corresponding anisotropic rotational Smoluchowski equation~\cite{Meirovitch2019} reads
\begin{equation}\label{eq:RotationalDiffusion}
  \partial_t f = \div \left[ \mathcal D \left( \nabla f + \beta f \nabla U \right) \right],
\end{equation}
where $\beta > 0$ is a constant controlling the strength of the drift induced by the orienting potential.
Here, $\div(\mathcal D \nabla f)$ describes anisotropic rotational diffusion, whereas $\beta \div(\mathcal D(f\nabla U))$ accounts for the potential-induced drift.

Tarroni and Zannoni derived an explicit frequency domain representation of the rotational diffusion operator by computing the entries of the corresponding operator matrix.
The framework developed here provides an alternative, modular construction from the underlying differential and algebraic operations.
In particular, the anisotropic Smoluchowski operator is assembled by computing the gradient and divergence in the frequency domain, cf.~\Cref{sec:VectorFields}, and by weighting the vector field components with the corresponding diffusion coefficients.
This avoids a problem-specific derivation of the complete operator matrix and allows the same procedure to be applied directly to different diffusion tensors and orienting potentials.

We now simulate the relaxation of an equilibrium rotation distribution after a sudden change of the preferred direction of the orienting medium.
For a unit vector $\vec n\in\mathbb S^2$, we consider the potential
\begin{equation*}
  U_{\vec n}(\mat R) = -\frac{1}{2} \left( 3 (\vec n^{\top} \mat R \vec z)^{2}-1 \right) -0.3 \left( (\vec n^{\top} \mat R \vec x)^2 - (\vec n^{\top} \mat R \vec y)^2 \right).
\end{equation*}
The initial probability density is chosen as
\begin{equation*}
  f_0(\mat R) = \frac1C \, \e^{-\beta U_{\vec z}(\mat R)}, \qquad C = \int_{\SO3} \e^{-\beta U_{\vec z}(\mat Q)} \d{\mu(\mat Q)}.
\end{equation*}
At $t=0$, the preferred direction is changed from $\vec z$ to $\mat R_{\vec y}(70^\circ)\vec z$, and \Cref{eq:RotationalDiffusion} is subsequently solved with the potential $U_{\mat R_{\vec y}(70^\circ)\vec z}$.

We use $(D_x,D_y,D_z)=(0.15,0.4,1)$, $\beta=2$, and approximate the solution by a harmonic expansion of bandwidth $L=32$.
The equation is integrated over $t\in[0,1]$ using $3000$ explicit Euler steps.
Selected snapshots of the resulting relaxation process are shown in \Cref{fig:RotationalDiffusion}.
The snapshots illustrate the transient rearrangement of the rotation distribution toward the new equilibrium distribution.
A video of the full time evolution underlying these snapshots is available in the accompanying repository, see the Code Availability section below.

\begin{figure}[H]
  \centering
  \begin{subfigure}[t]{0.3\textwidth}
    \centering
    \includegraphics[width = \textwidth]{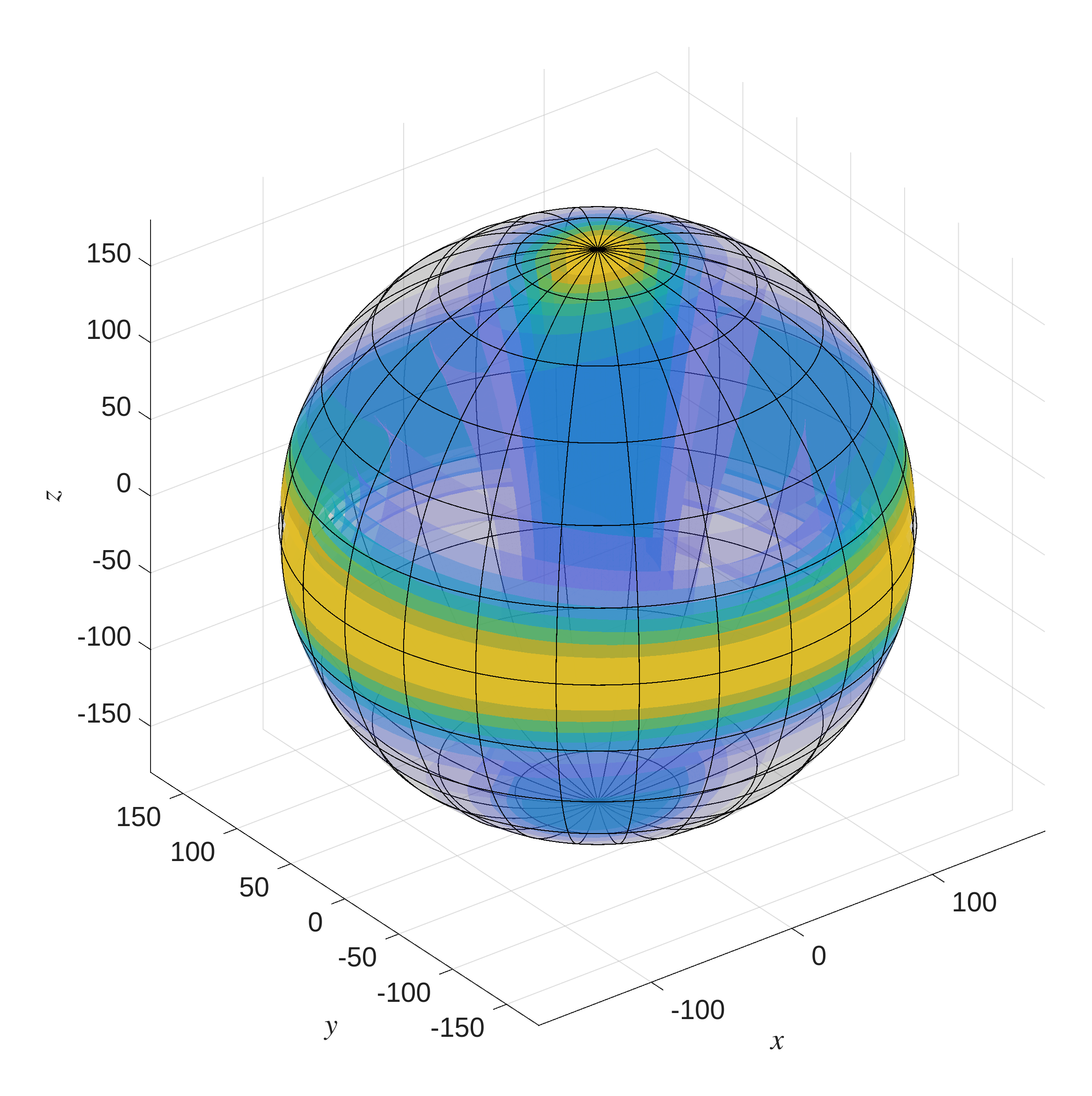}
    \subcaption{Rotation density ($t=0$)}
  \end{subfigure}
  \begin{subfigure}[t]{0.3\textwidth}
    \centering
    \includegraphics[width = \textwidth]{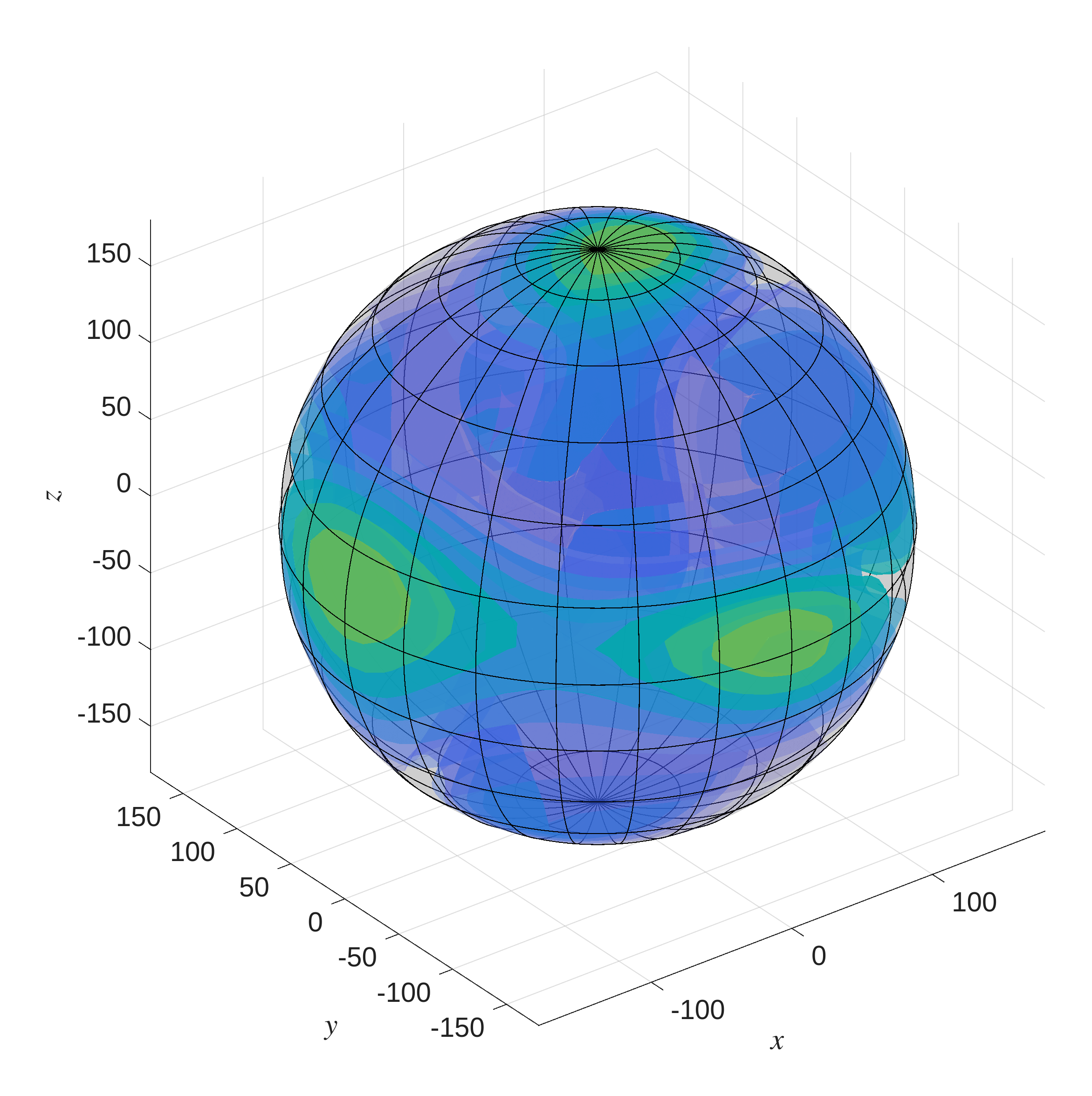}
    \subcaption{$600$ time steps ($t=0.2$)}
  \end{subfigure}
  \begin{subfigure}[t]{0.3\textwidth}
    \centering
    \includegraphics[width = \textwidth]{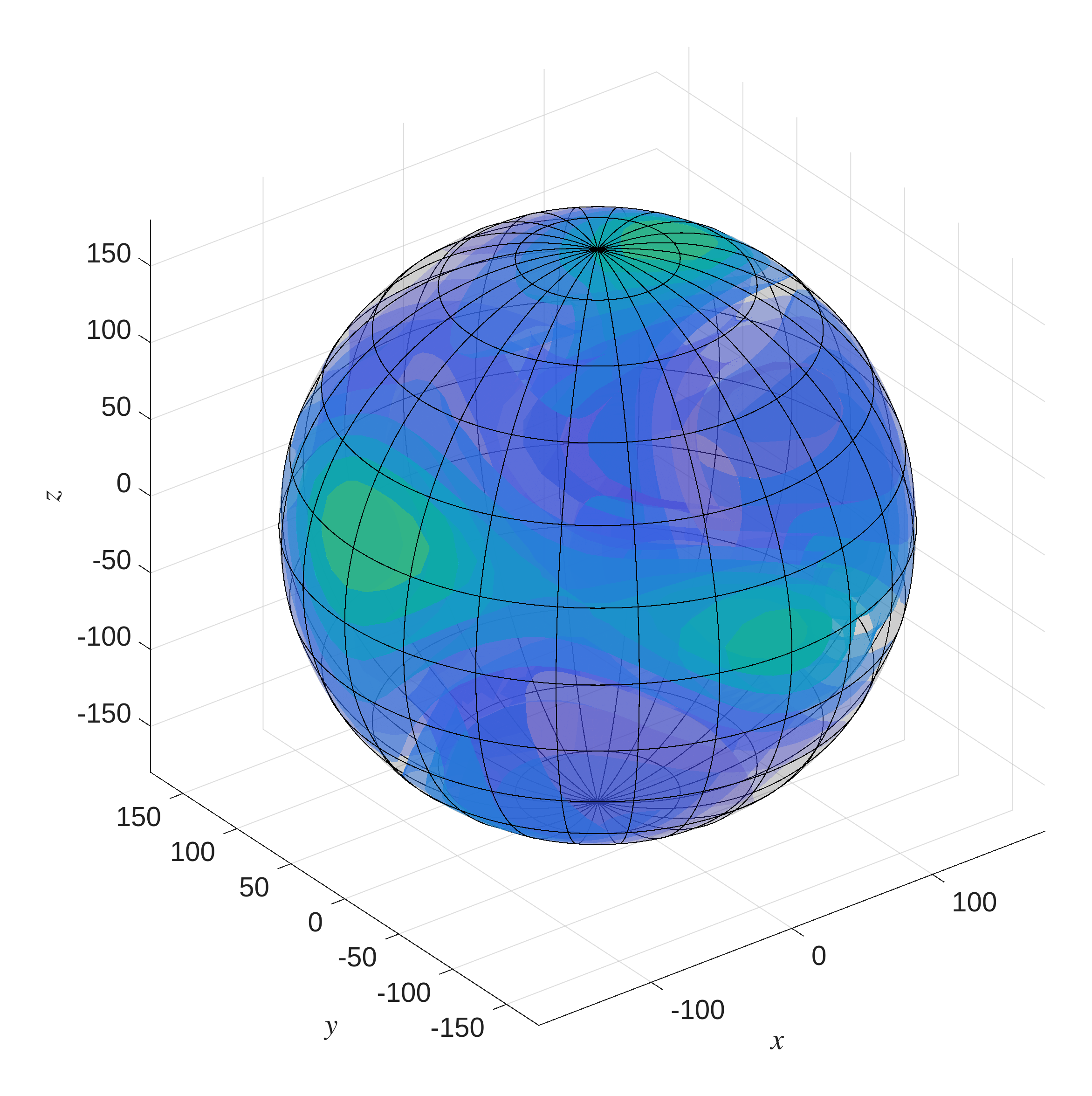}
    \subcaption{$1200$ time steps ($t=0.4$)}
  \end{subfigure}
  \begin{subfigure}[t]{0.3\textwidth}
    \centering
    \includegraphics[width = \textwidth]{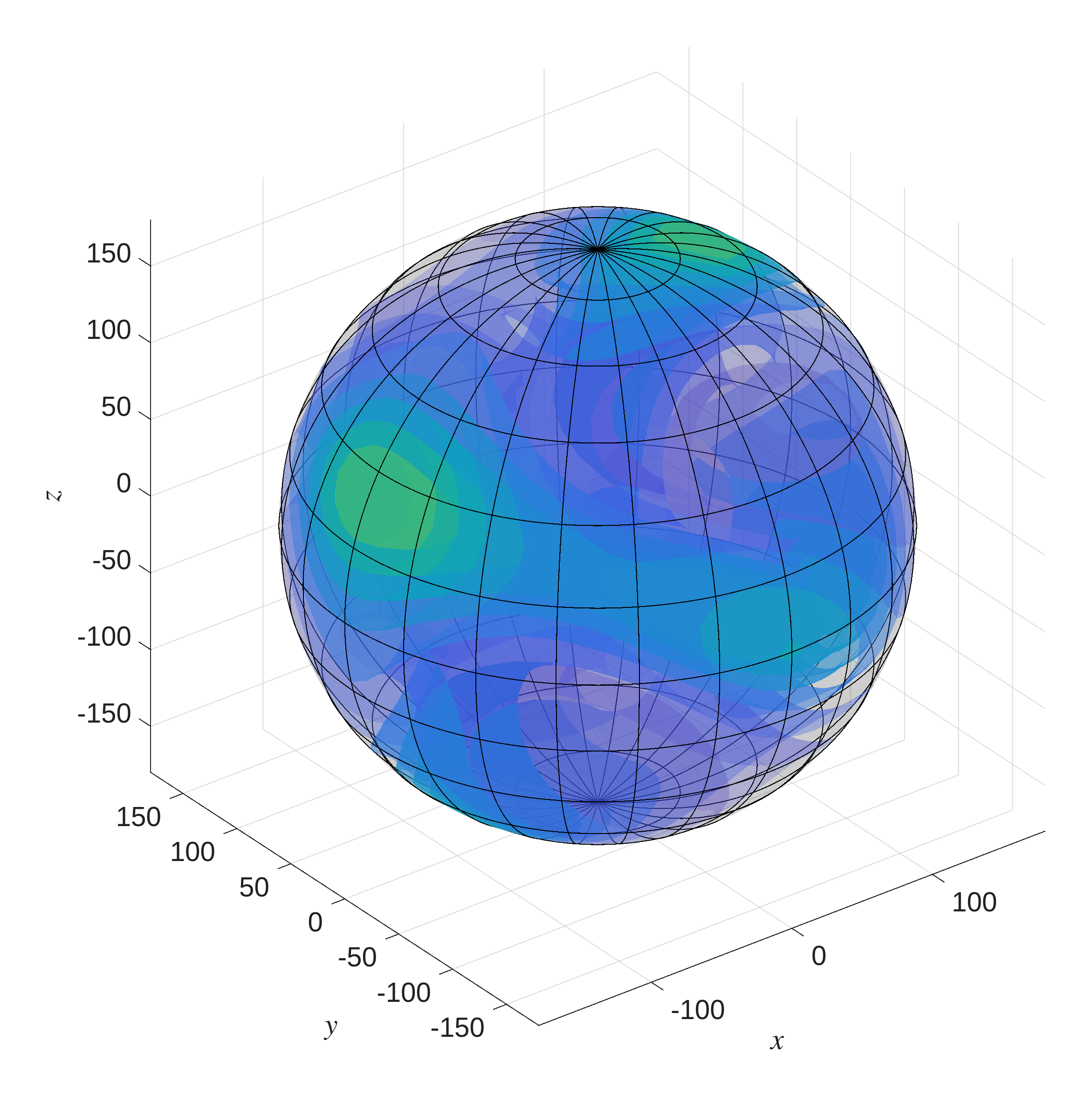}
    \subcaption{$1800$ time steps ($t=0.6$)}
  \end{subfigure}
  \begin{subfigure}[t]{0.3\textwidth}
    \centering
    \includegraphics[width = \textwidth]{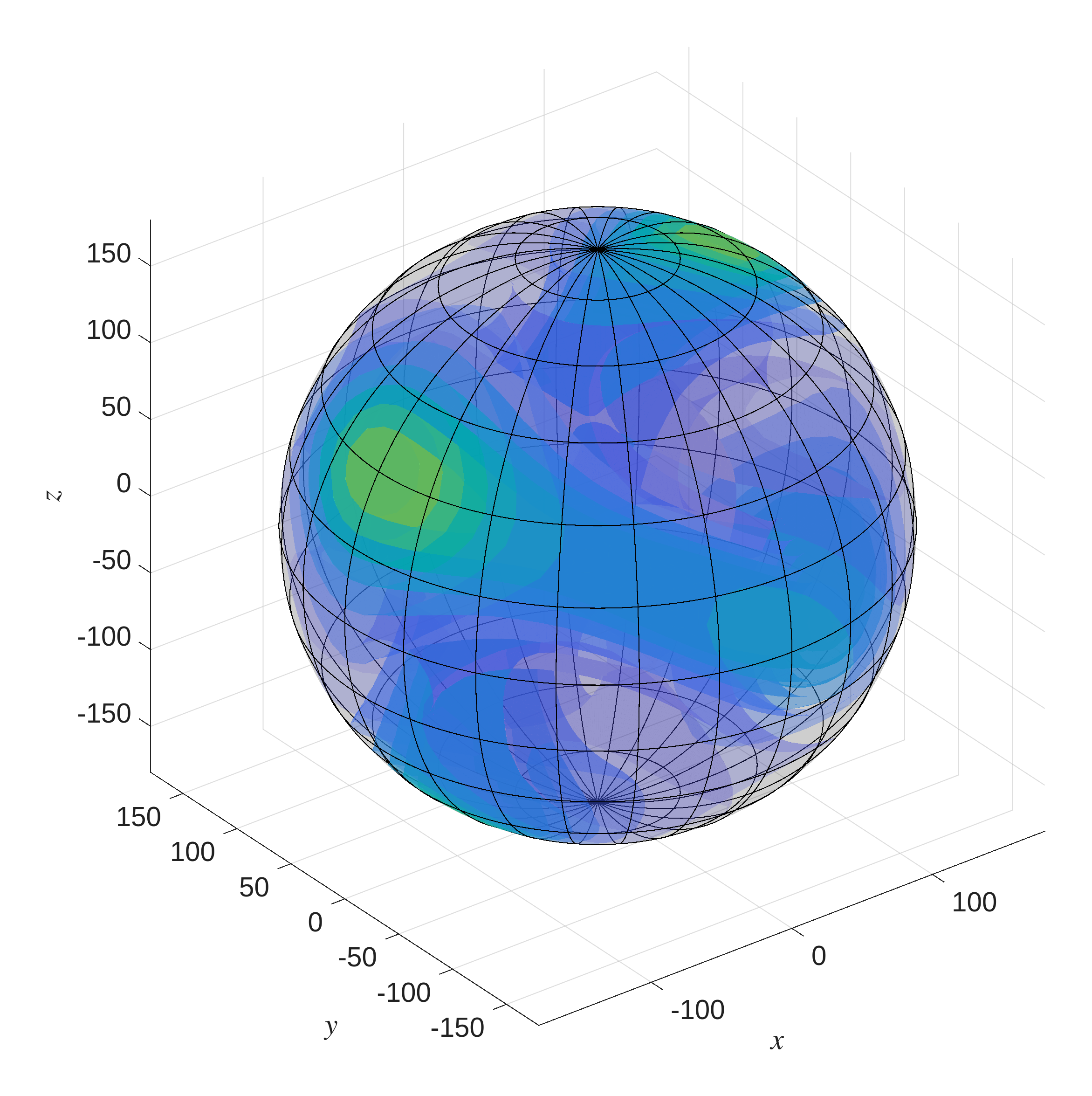}
    \subcaption{$2400$ time steps ($t=0.8$)}
  \end{subfigure}
  \begin{subfigure}[t]{0.3\textwidth}
    \centering
    \includegraphics[width = \textwidth]{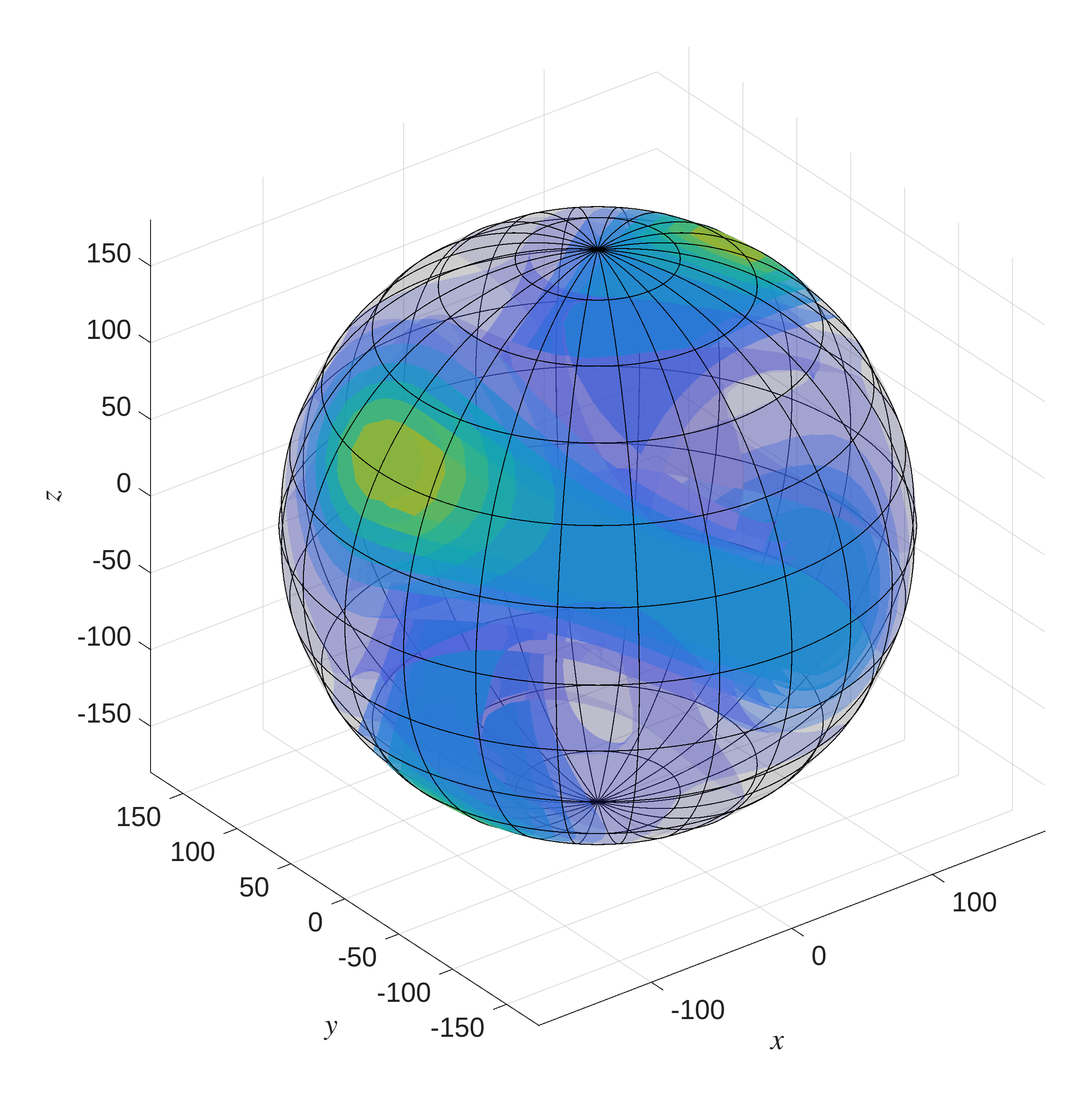}
    \subcaption{$3000$ time steps ($t=1$)}
  \end{subfigure}
  \caption{ Snapshots of the anisotropic rotational diffusion on $\SO3$ in axis-angle representation.
    The initial distribution is the equilibrium density associated with the potential $U_{\vec z}$. At $t=0$, the preferred direction of the orienting potential is changed to $R_{\vec y}(70^\circ)\vec z$, and the solution of \cref{eq:RotationalDiffusion} relaxes toward the corresponding new equilibrium.
    The panels show the rotation density after selected numbers of explicit Euler time steps.}\label{fig:RotationalDiffusion}
\end{figure}


\subsection*{Code Availability}
All algorithms presented in this paper are implemented in the MATLAB toolbox \texttt{MTEX}~7.0~\cite{MTEX}.
The numerical experiments were carried out using this implementation.
The corresponding scripts and the video accompanying \Cref{sec:BrownianMotion} are available at \url{https://github.com/mtex-toolbox/mtex-paper/tree/master/HarmonicVectorFieldRepresentationOnSO3}.

\bibliography{literatur}

\end{document}